\documentclass[10pt,a4paper,final]{article}
\usepackage[T1]{fontenc}
\usepackage[utf8]{inputenc}
\usepackage{amsmath,amsfonts,amsthm,mathrsfs,amssymb,cite,bm,mathtools}
\usepackage{graphicx,appendix}
\usepackage{showkeys}
\usepackage{float}
\usepackage{placeins}
\usepackage{color}
\usepackage{indentfirst}
\usepackage{booktabs,multirow}
\usepackage{longtable}
\usepackage[affil-it]{authblk}
\usepackage{caption}
\usepackage{subcaption}

\usepackage{algorithm,algorithmic}
\usepackage{enumitem} 

\numberwithin{equation}{section}
\newtheorem{theorem}{Theorem}[section]
\newtheorem{lemma}[theorem]{Lemma}

\newtheorem{definition}[theorem]{Definition}
\newtheorem{remark}[theorem]{Remark}

\newtheorem{example}[theorem]{Example}
\usepackage{amsopn}

\def\bbC{\mathbb{C}} \def\bbR{\mathbb{R}} 
 
\def\bbS{\mathbb{S}}
\def\mcB{\mathcal{B}} \def\mcS{\mathcal{S}}

\def\G{\Gamma}   \def\O{\Omega}
\def\p{\partial}
\def\q{\quad}   
\def\forany{\,\forall}
\def\mbd{\mathbf{d}}
\def\mbn{\mathbf{n}}
\def\mbx{\mathbf{x}}
\def\mbz{\mathbf{z}}
\def\btau{\boldsymbol{\tau}}

\allowdisplaybreaks[4]
          \title{\bfseries  Factorization and monotonicity methods for reconstructing an impenetrable obstacle in inverse biharmonic scattering}
		\author[1]{Tielei Zhu}
		\author[2]{Zhihao Ge}
                \author[3]{Bangmin Wu\thanks{Corresponding author: bmwu\_math@xju.edu.cn}}
                \author[4]{Dongwoo Sheen}
                
              \affil[1]{College of Mathematics and Information Science, Henan University of Economics and Law, Zhengzhou, Henan, 450046, PR China.}
                \affil[2]{School of Mathematics and Statistics, Henan University, Kaifeng, Henan, 475004, PR China.}
                \affil[3]{College of Mathematics and Systems Science, Xinjiang University, Urumqi, Xinjiang, 830017, PR China.}
                \affil[4]{Department of Mathematics, Seoul National University,  Seoul 08826, Korea.}
		\date{\today}
\begin{document}
	\maketitle
	\vspace{.2in}

        \begin{abstract} 
          We study the recovery of an impenetrable obstacle with a Dirichlet or Neumann boundary condition from far-field measurements. The   associated data-to-pattern operators are noninjective, because the biharmonic scattered fields can be decomposed into propagative and evanescent components. We derive new symmetric factorizations of the far-field operator and analyzie the middle operators on the closures of the ranges of the adjoint data-to-pattern operators to establish range identities. These identities yield factorization criteria for the support of the obstacle when the wavenumber is not an associated transmission eigenvalue. Furthermore, we derive monotonicity criteria without excluding  transmission eigenvalues.  Numerical experiments for Dirichlet obstacles illustrate the reconstruction methods and show their performance with noisy data and for wavenumbers near a clamped transmission eigenvalue.
\end{abstract}

\vspace{.2in}\noindent{\bf Keywords:} Factorization method, monotonicity method, inverse scattering, biharmonic wave equation, shape reconstruction.
  
\vspace{.2in}\noindent{\bf AMES subject classifications}: 35P25, 35J30, 35R30
  
\section{Introduction}\label{sec:intro}
This paper develops factorization and monotonicity methods for inverse biharmonic scattering problems. Such inverse problems arise in the nondestructive testing and imaging of thin elastic plates, with applications to structural health monitoring, and have therefore attracted considerable academic attention.
Although scattering problems for acoustic, elastic, and electromagnetic waves have been extensively investigated, the fourth-order structure of the biharmonic wave operator leads to an evanescent component, that cannot be determined by far-field measurements. As a result, the associated data-to-pattern operators are not injective, and range characterizations require an analysis of this evanescent component.

Extensive research has been devoted to numerical methods for  inverse scattering problems of biharmonic waves. In \cite{chang2023optimization}, Chang  \& Guo introduced a novel optimization method to detect a clamped obstacle from the near-field data consisting of both the scattered field and its Laplacian, which is related to the decomposition methods discussed in \cite{colton2019inverse}.
In \cite{bourgeois2020linear}  Bourgeois   \&  Recoquillay developed a near-field linear sampling method (LSM) for the reconstruction of an obstacle with Dirichlet or free plate boundary condition. Guo \emph{et al.} subsequently  introduced a far-field LSM to identify the Dirichlet obstacle in \cite{guo2024obstacle},
whereas Harris \emph{et al.} \cite{harris2025sampling} used a new factorization of far-field operator to relax  the assumption on the Dirichlet eigenvalue of the scatterer, and proposed an extended sampling method to find the Dirichlet obstacle. In addition to these aforementioned  methods, direct sampling methods have also been applied to shape identification problems in \cite{harris2025direct, zhu2025direct}. 

Factorization methods, which provide a rigorous characterization of the obstacle by using the spectral data of the far-field operator, have been proven to be powerful tools for a wide class of inverse problems, since Kirsch \cite{kirsch1998charac} introduced the $(F^{*}F)^{1/4}$-method for the reconstruction of the obstacle with Dirichlet or Neumann boundary conditions in inverse acoustic scattering problems. Subsequently, this method was further developed into a novel functional analysis framework, namely  the $F_{\#}^{1/2}$-method \cite{kirsch2008factorization}. 
Regarding the applications to  various time-harmonic inverse scattering problems, we refer to \cite{arens2001linear, alves2002far, hu2013some}
for elastic scattering  problems, to \cite{kirsch2004factorization}
for electromagnetic scattering problems, and to \cite{kirsch2012factorization, yin2016near, wu2023reconstruction} for fluid-solid
interaction scattering problems. Recently, the factorization methods have been applied to absorbing penetrable obstacles from far-field biharmonic data  \cite{ayala2025factorization} and to a simply supported obstacle (with Poisson ratio 1) from near-field point-source measurements \cite{harris2026factorization}. Besides, this factorization method has been extended to time-domain inverse scattering problems, such as acoustic scattering by obstacles with Dirichlet or Robin boundary conditions (see \cite{cakoni2019factorization, haddar2020time}), and even to the inverse Stokes problem \cite{lechleiter2013factorization}.
We remark that the factorization of far-field operator in the factorization method is closely related to the monotonicity method developed  in \cite{albicker2020monotonicity}.

The monotonicity-based method was originally introduced for the inverse conductivity problem \cite{gebauer2008localized, harrach2013monotonicity}
and has been  successfully applied to various inverse problems (see \cite{brander2018monotonicity,  lin2022monotonicity}).
For the inverse medium scattering problems, Griesmaier \& Harrach
initiated the study of the monotonicity method
\cite{griesmaier2018monotonicity}, which is based on a special ordering (an extension of the Loewner order) between a given far-field operator and virtual probing operators to design the reconstruction  algorithm. This approach was later extended to the case of scattering by an impenetrable obstacle by  Albicker \& Griesmaier \cite{albicker2020monotonicity}. In \cite{furuya2021remarks}, Furuya later established a
general functional analysis theorem for the monotonicity method in inverse acoustic scattering problems. Further developments of the monotonicity method can be found in \cite{eberle-blick2025monotonicity, harrach2026monotonicity}.

The main contributions are as follows.  
	\begin{itemize}[leftmargin=*,noitemsep,topsep=0pt]
		\item \emph{Factorization criteria.}
		Although the data-to-pattern operators are not injective, we rigorously verify the strict sign properties of the middle operators on the closures of adjoint data-to-pattern operators  by  using the annihilators of the corresponding kernels. This yields  characterizations of the support of the obstacle by the $(F_D^*F_D)^{1/4}$- and $F_{D,\#}^{1/2}$-methods, when the wavenumber is not an associated transmission eigenvalue. In the Neumann case, we further assume that the wavenumber lies outside $\mathcal K_0$ .
		\item \emph{Monotonicity criteria.}
		Using the factorizations of the far-field operator, we establish, for both the Dirichlet and Neumann cases, inclusion tests, that provides reconstruction schemes based on the monotonicity relation between the real part of the far-field operator and probing operators. These schemes are valid for every positive wavenumber without any corresponding transmission eigenvalue assumption; in the Neumann case, however, wavenumbers  in $\mathcal K_0$ must additionally  be excluded.
\end{itemize}

The remainder of this paper is organized as follows. Section 2 formulates the direct and inverse biharmonic scattering problems and introduces relevant analytical tools. In Section 3, we then study the factorizations of the far-field operator and develop factorization methods for Dirichlet and Neumann obstacles. Based on these factorizations, Section 4 presents the  monotonicity method. Finally,  numerical  experiments are presented in Section 5 to show the reconstruction results of the proposed algorithms.

\section{Problem formulation}
\label{sec:problem}
Throughout the paper, we assume that  $D\subset\bbR^2$ is a
bounded domain  with $C^3$-smooth boundary $\p D$,
representing an impenetrable obstacle such that its complement $D^{\rm c}\coloneqq \bbR^2\setminus\overline{D}$ is connected. We denote by $\mbn=(n_1,n_2)^{\top}$ the outward unit normal to the boundary $\p D$. The incident and the scattered fields
are denoted by $u^{\rm in}$ and  $u^{\rm sc}$, respectively,
and the total field by $u=u^{\rm in}+u^{\rm sc}$.
\emph{The direct biharmonic wave scattering problem} for the impenetrable obstacle is to find the scattered field $u^{\rm sc} \in  H^2_{\mathrm{loc}}(D^{\rm c})$ such that
\begin{subequations}\label{eq:bvp}
	\begin{align}
		\label{eq:bvp-1}
		\Delta^2u^{\rm sc} - \kappa^4u^{\rm sc} &= 0 \quad \mathrm{ in}\; D^{\rm c},\\
		\label{eq:bvp-2}
		(\mcB_1 ( u^{\rm sc}+u^{\rm in}),\mcB_2 ( u^{\rm sc}+u^{\rm in})) &= {\bm 0} \quad \mathrm{on}\; \p D,\\
		\label{eq:bvp-3}
		\lim\limits_{r\to\infty}\sqrt{r}\left(\frac{\p u^{\rm sc}}{\p r}
		-\mathrm{i}\kappa u^{\rm sc} \right)&=0\quad r=| \mbx |,
	\end{align}
\end{subequations}
where $\kappa>0$ is the wavenumber.
Notice that \eqref{eq:bvp-2} covers both cases of Dirichlet
boundary condition (corresponding to a clamped plate)
if $(\mcB_1,\mcB_2)=(I,\p_{\mbn})$ and 
Neumann boundary condition (corresponding to a free plate)
if $(\mcB_1,\mcB_2)=(M,N)$
(see \cite{chen2010boundary}), where the bending moment, $M$, and
effective shear operators, $N$, are defined by
\begin{equation*}
	M v = \nu \Delta v + (1-\nu) M_0 v \quad
	\mathrm{ and } \quad
	N v = -\p_{\mbn}\Delta v - (1-\nu) \p_{\btau} N_0 v,
\end{equation*}
for the Poisson ratio $\nu\in[0,1/2)$. Here,
$\p_{\btau}$ denotes the tangential derivative, 
and $M_0=\p^2_{\mbn}$ and $N_0=\p_{\mbn}\p_{\btau}$ are auxiliary boundary operators defined by
\begin{equation*}
M_0 v = n_1^2 \dfrac{\partial^2 v }{\partial x_1^2} + 2n_1n_2 \dfrac{\partial^2v }{\partial x_1\partial x_2}
+ n_2^2 \dfrac{\partial^2 v }{\partial x_2^2}
\end{equation*}
and
\begin{equation*}
N_0 v = -\Big\{  \big( \dfrac{\partial^2 v }{\partial x_1^2}-  \dfrac{\partial^2 v }{\partial x_2^2} \big)n_1n_2  - \dfrac{\partial^2v }{\partial x_1\partial x_2}(n_1^2 - n_2^2)
\Big\}.
\end{equation*}
  Sommerfeld radiation condition \eqref{eq:bvp-3},
  which holds uniformly with respect to all directions $\hat{\mbx} = \frac{\mbx}{|\mbx|} \in  \bbS:=\{\mbx\in\mathbb{R}^2:|\mbx|=1 \}$, yields the far-field pattern $u^{\infty}$ of the scattered field $u^{\rm sc}$ such that
\begin{equation}
  \label{eq:asy}
  u^{\rm sc}(\bm{x}) = \dfrac{e^{\pi \mathrm{i}/4}}{\sqrt{8\pi\kappa}}
  \dfrac{e^{\mathrm{i}\kappa|\bm{x}|}}{|\bm{x}|^{1/2}}
u^{\infty}(\hat{\bm{x}}) + O\left(\dfrac{1}{|\bm{x}|^{3/2}} \right) \quad |\bm{x}|\to\infty.
\end{equation}
In particular, if the dependence of $u^{\rm sc}$ on $\mbd$ in
\eqref{eq:bvp} needs to be emphasized, the scattering wave
and the far-field pattern are denoted by $u^{\rm sc}(\mbx,\mbd)$ and
$u^{\infty}(\hat{\mbx},\mbd),$ respectively.

Henceforth, a solution of \eqref{eq:bvp-1} satisfying the radiation condition \eqref{eq:bvp-3} will be referred to as a \emph{radiating solution} of the biharmonic equation in $D^c$.
The radiating fundamental solution $G$  of the biharmonic wave equation is given by
\begin{equation}
  \label{eq:funda-bi}
  G(\bm{x},\bm{y}) = \frac{1}{2\kappa^{2}}
  \left[\Phi_{\mathrm{i}\kappa}(\bm{x},\bm{y})- \Phi_{\kappa}(\bm{x},\bm{y}) \right],
\end{equation}
where $\Phi_{\kappa}(\bm{ x },\bm{ y })=\frac{\mathrm{ i }}{4} H^{(1)}_0(\kappa | \bm{ x }-\bm{ y } |)$
and $H^{(1)}_0$ is the Hankel function of the first kind of the order zero.
From the asymptotic behaviors \cite{dong2024novel} of the Hankel function $H^{(1)}_0$, we see that
\begin{equation}
  \label{eq:asy-G}
  G(\bm{ x } , \bm{ y }) = \dfrac{e^{\pi \mathrm{i}/4}}{\sqrt{8\pi\kappa}}
  \dfrac{e^{\mathrm{i}\kappa|\bm{x}|}}{|\bm{x}|^{1/2}}
  \left(  -\frac{1}{2 \kappa^2}e^{-{\rm i}\kappa \hat{\bm{ x }} \cdot \bm{ y }}  \right)
  +  O\left(\dfrac{1}{|\bm{x}|^{3/2}} \right) \quad |\bm{x}|\to\infty.
\end{equation}

The  well-posedness of the exterior boundary value problem \eqref{eq:bvp-1}-\eqref{eq:bvp-3}  has been established in \cite{bourgeois2020well} and \cite{wu2025obstacle}. More precisely,  for the Dirichlet obstacle, the result holds for all positive wavenumbers as established by the variational method or boundary integral method   and for the Neumann obstacle, it was obtained  by the variational method for all  positive wavenumbers 
 except $\kappa\in \mathcal{K}_0$, where $\mathcal{K}_0$ is a discrete set consist of positive numbers $\kappa_n$ ($n\in \mathbb{N}$) accumulating at $+\infty$. 
Throughout this paper, \emph{ we assume $\kappa\not\in \mathcal{K}_0$ for the case of Neumann obstacle.}

\emph{The inverse scattering problem} we study is to recover the location and shape of the impenetrable obstacle $D$ from knowledge of the far-field patterns $u^{\infty}(\hat{\mbx},\mbd)$ for all $\hat{\mbx},\bm{ d}\in \bbS$, where  $u^{\infty}(\hat{\mbx},\mbd)$ is the far-field pattern of the scattered field  $u^{\rm sc}(\mbx,\mbd)$ with  $u^{\rm in}(\mbx,\mbd)= e^{{\rm i}\kappa \mbx\cdot \mbd}$  as the incident field. Note that the uniqueness for the inverse problems based on far-field patterns in the Dirichlet obstacle case have been discussed by Dong \& Li \cite{dong2024uniqueness}
	and by Wu \& Yang \cite{wu2025obstacle}. The  Neumann obstacle case, however,  remains an open problem. Our factorization methods also imply the uniqueness for this inverse problem, provided that the wavenumber is not a corresponding transmission eigenvalue.
	
For latter use, we introduce the integration by parts in the biharmonic wave equation (see \cite{Hsiao2021boundary}).
\begin{lemma}[Green's formulas]
  \label{lem:green}
  Let $\Omega$ be a bounded and $C^2$-smooth domain, and denote $\bm{ n }$ by the outward unit normal on $\partial \Omega$. Then
for $v\in H^2(\Omega,\Delta^2):=\{v \in H^2(\Omega):\Delta^2v \in L^{2}(\Omega)\}$ and $w \in H^2(\Omega)$, the following formula holds:
\begin{equation} \label{eq:green-1}
  \int_{ \Omega } (\Delta^{2} v) w \;\mathrm{d}\bm{x}
  = \int_{ \Omega }  a(v,w) \;\mathrm{d}\bm{x}
  - \int_{ \partial \Omega } \left[(Mv)\partial_{\bm{n} } w + (Nv) w \right] \;\mathrm{d}s,
\end{equation}
where
\begin{equation*}
  a(v,w)= \nu  \Delta v \Delta w  +
  (1-\nu)\left( \dfrac{\partial^2 v}{\partial x_1^2}\dfrac{\partial^2 w}{\partial x_1^2}
    + 2 \dfrac{\partial^2 v}{\partial x_1 \partial x_2 }\dfrac{\partial^2 w}{\partial x_1\partial x_2  }
    +\dfrac{\partial^2 v}{\partial x_2^2}\dfrac{\partial^2 w}{\partial x_2^2}
     \right).
\end{equation*}
For $v,w \in H^2(\Omega,\Delta^2)$, the following formula holds:
\begin{equation} \label{eq:green-2}
  \int_{ \Omega }\Big[ w (\Delta^{2} v) -  v(\Delta^{2} w)\Big] \;\mathrm{d}\bm{x} =   \int_{ \partial \Omega }\Big\{
   \left[(Mw)\partial_{\bm{n} } v + (Nw) v \right]
  - \left[(Mv)\partial_{\bm{n} } w + (Nv) w \right]\Big\} \;\mathrm{d}s.
\end{equation}
\end{lemma}
Using the above Green's formulas and the fundamental solution $G$, we can obtain the representations for  the radiation solution and the associated far-field pattern  of the biharmonic wave equation.
\begin{lemma}[see Lemma 2.2 in \cite{zhu2025direct}]
  \label{lem:rep}
    Let $\Omega$ be a bounded and $C^2$-smooth domain. 
  If $w\in H^2_{\mathrm{loc}}(\Omega^{c})$ is the radiation solution of the biharmonic wave equation in $\Omega^c$, 
  then   $w$ is given by
  \begin{align}
    \notag
    w(\bm{x})  =&
                \int_{\Gamma}\Big\{
                  [ -G(\bm{x},\bm{y})\partial_{\bm{n}}\Delta w(\bm{y})
                  + \partial_{\bm{n}(\bm{y})}G(\bm{x},\bm{y})\Delta w(\bm{y}) ]\\
    \label{eq:Green-rep}
              &\qquad\quad  - [ -\partial_{\bm{n}(\bm{y})}\Delta_{\bm{y}} G(\bm{x},\bm{y})    w(\bm{y})
                    + \Delta_{\bm{y}} G(\bm{x},\bm{y})   \partial_{\bm{n}} w(\bm{y}) ]
                    \Big\} \;\mathrm{d}s(\bm{y}) \quad \bm{x}\in D^{c},  
  \end{align}
  and the associated far-field pattern $w^{\infty}$ are given by
\begin{align}
    \notag
            w^{\infty}(\hat{\bm x}) = & -\frac{1}{2\kappa^{2}}
           \int_{\partial D}\Big\{
           [  e^{-\mathrm{i}\kappa\hat{\bm x}\cdot \bm{y}}N w(\bm{y})
           + (\partial_{\bm{n}(\bm{y})}e^{-\mathrm{i}\kappa\hat{\bm x}\cdot \bm{y}})  M w(\bm{y}) ] \\
           \label{eq:far-rep}
             & \qquad\qquad\quad   - [ (N_{\bm{y}}e^{-\mathrm{i}\kappa\hat{\bm x}\cdot \bm{y}})    w(\bm{y})
                    + (M_{\bm{y}}e^{-\mathrm{i}\kappa\hat{\bm x}\cdot \bm{y}})  \partial_{\bm{n}} w(\bm{y}) ]
               \Big\} \;\mathrm{d}s(\bm{y}) \quad \hat{\bm x}\in \mathbb{S}.
\end{align}
\end{lemma}

The following lemma is significant in the analysis of the factorization methods.
\begin{lemma}[see Lemma 2.3 in \cite{zhu2025direct}]
  \label{lem:sign}
  Let $\Omega$ be a bounded and $C^2$-smooth domain. If  $w \in H_{\mathrm{ loc } }^2(\Omega^c)$ is the radiation solution of the biharmonic wave equation in $\Omega^c$, then the following formula holds:
  \begin{equation*}
    \mathrm{Im} \int_{\partial \Omega}[\overline{w}\, N w +  \partial_{\bm{n}}\overline{w}\, M w] \;\mathrm{d}s
    = \frac{\kappa^2}{4\pi}\int_{ \mathbb{S}   } |w^{\infty}|^2 \;\mathrm{d}s
  \end{equation*}
  where $w^{\infty}$ is the far-field pattern of $w$.
\end{lemma}

\section{Factorization methods}
\label{sec:factorization}
In this section, we present two factorization methods, i.e., $(F^*F)^{1/4}$- and $F^{1/2}_{\#}$-methods for the inverse biharmonic scattering problem in the case of the Dirichlet or Neumann obstacle. We recall classic range identities \cite{kirsch2008factorization, furuya2021remarks} and apply them to obtain the characterization for the support of the obstacle.

\begin{theorem}[$(F^*F)^{1/4}$-method] \label{thm:range-1}
	Let $H$ be a Hilbert space, with the identification $H^*\equiv
	H,$
	$X$ a reflexive Banach space and let $F\colon H\to H$ be a compact operator  such that
		$   F=BAB^{*}$,
	where $B\colon X\to H$ and $A\colon X^{*}\to X$ are linear bounded operators. Assume that
	\begin{enumerate}
		\item[(1)] $ A =   A_{0} + C$ with a compact operator $C$ and a self-adjoint operator $A_{0}$
		which satisfies the  coercivity condition  on the range ${\rm Ran}(B^{*})$ of the operator $B^{*}$, i.e.,
		\begin{equation*}
			| \langle \varphi, A_0\varphi\rangle | \geq c \| \varphi \|_{X^{*}}^2 \quad \text{ for all }  \quad \varphi \in {\rm Ran}(B^{*}).
		\end{equation*}
		
		\item[(2)] ${\rm Im}\, \langle \varphi,A\varphi\rangle \not= 0$ for all $\varphi \in \overline{{\rm Ran}(B^{*})}$ with $\varphi\neq0$.
		
		\item[(3)] $F$ is injective and $I+{\rm i}rF$ is unitary for some $r>0$.
	\end{enumerate}
	Then the ranges of $B$ and $(F^*F)^{1/4}$ coincide.
\end{theorem}

\begin{theorem}[$F_{\#}^{1/2}$-method] \label{thm:range-2}
	Let $X\subset H\subset X^{*}$ be a Gelfand triple with a Hilbert
	space $H,$ with the identification $H^*\equiv H,$ and a
	reflexive Banach space $X$ with dense embeddings.
	Furthermore, let $Y$ be a Hilbert space, with the identification $Y^*\equiv Y,$ and let $F\colon Y\to Y$,  $G\colon X\to Y$,  $T\colon X^{*}\to X$ be linear bounded operators such that
	$
	F=GTG^{*}.
	$
	Assume that
	\begin{enumerate}
		\item[(1)] $G$ is compact with dense range in $Y$.
		
		\item[(2)] There exists $t \in [0,2\pi]$ such that ${\rm Re}\, (e^{{\rm i}t}T) =   T_{0} + K$, where $K$ is a self-adjoint compact operator and   $T_{0}$ is a positive coercive operator, i.e., there exists a constant $C>0$ such that
		\begin{equation*}
			\langle\varphi, T_{0}\varphi\rangle \geq C  \| \varphi \|_{X^{*}}^2 \quad\text{for all}\quad \varphi \in  X^{*},  
		\end{equation*}
		where $\langle \cdot ,\cdot \rangle$ denotes the duality pairing between $X^{*}$ and $X$.
		
		\item[(3)]  ${\rm Im}\, \langle \varphi,T\varphi\rangle > 0$ for all $\varphi \in \overline{{\rm Ran}(G^{*})}$ with $\varphi\not=0$.
	\end{enumerate}
	Then the ranges of $G$ and $F_{\#}^{1/2}$ coincide,
	where $F_{\#}\coloneqq | {\rm Re}\,  F| + | {\rm Im}\,  F|$, and $ {\rm Re}\,  F$ and $ {\rm Im}\,  F$ are defined by
	\begin{equation*}
		{\rm Re}\,  F \coloneqq \frac{F+ F^{*}}{2} \quad\text{and} \quad
		{\rm Im}\,  F \coloneqq \frac{F- F^{*}}{2 {\rm i}}
	\end{equation*}
	respectively.
\end{theorem}

\subsection{Properties of the far-field operator}
\label{sec:far-field}
Given the far-field pattern $u^\infty(\cdot,\mbd)$ corresponding to the incident plane wave with direction $\mbd$, the \emph{far-field operator}
$F_{D}: L^2(\bbS) \to L^2(\bbS)$ is defined by
\begin{equation*}
	F_{D} g(\hat{\mbx}) \coloneqq \int_{ \bbS   } u^{\infty}(\hat{\mbx}, \mbd) g({\mbd}) \,\,{\rm d}s(\mbd)
	\quad \hat{\mbx} \in \bbS\q \forany g\in L^2(\bbS).
\end{equation*}

Before applying the range identity in the $(F^*F)^{1/4}$-method, we need a comprehensive investigation for the properties of the operator $F_D$.
For the Dirichlet problem, we introduce \emph{a clamped transmission eigenvalue problem}:
\begin{equation}
	\label{eq:clamped}
	\begin{aligned}
		\Delta v - \kappa^2 v & = 0   \quad \mathrm{ in } ~ D^{\rm c}, \\
		\Delta w + \kappa^2 w& = 0   \quad \mathrm{ in } ~ D, \\
		(v+w, \p_{\mbn} (v +w) )  & = {\bm 0}  \quad \mathrm{ on } ~ \p D, 
	\end{aligned}
\end{equation}
where $v$ decays exponentially as the argument tends to infinity.  The wavenumbers $\kappa$ are called the \emph{clamped transmission eigenvalue} if \eqref{eq:clamped} has a non-trivial solution. The clamped transmission eigenvalue problem has been studied in \cite{harris2025transmission, harris2025existence} and the set of the associated eigenvalues is discrete under some assumptions. 
For the Neumann problem, we introduce \emph{a free transmission eigenvalue problem}:
\begin{equation}
	\label{eq:free}
	\begin{aligned}
		\Delta v - \kappa^2 v & = 0   \quad \mathrm{ in } ~ D^{\rm c}, \\
		\Delta w + \kappa^2 w& = 0   \quad \mathrm{ in } ~ D, \\
		(M(v+w), N( v + w))  & = \bm{0}  \quad \mathrm{ on } ~ \p D, 
	\end{aligned}
\end{equation}
where $v$ decays exponentially as the argument tends to infinity.  The wavenumbers $\kappa$ are called  \emph{free transmission eigenvalues} if \eqref{eq:free} has a non-trivial solution. 
Let us state some properties of the operator $F_D$.
\begin{theorem}\label{thm:far-field}
	Let $F_D$ be the far-field operator associated with the Dirichlet or Neumann boundary value problem. Then, the following hold:
	\begin{enumerate}
		\item[(a)] $F_D$ is compact.
		\item[(b)]  For the Dirichlet problem, $F_D$ is injective   if $\kappa$ is not a clamped transmission eigenvalue. For the Neumann problem, $F_D$ is injective   if $\kappa$ is not a free transmission eigenvalue.
		\item[(c)] $I+ \frac{\rm i}{4\pi} F_{D}$ is unitary.
	\end{enumerate}
\end{theorem}
\begin{proof}
	We first remark that (a) and the Dirichlet case of (b) have been proved in \cite{harris2025direct, harris2025sampling}, respectively, and thus the proofs are omitted.
	
	(b) It suffices to consider the case of the Neumann obstacle. Given $g \in L^2(\bbS )$, define $v^{\rm in}(\mbx)$ by
	$
	v^{\rm in}(\mbx)= \int_{ \bbS   } e^{{\rm i}\kappa \mbx\cdot  \mbd} g(\mbd)  \,\,{\rm d} s(\mbd) \quad \mbx\in \bbR^2.
	$
	Let $v^{\rm sc}$ be the solution to \eqref{eq:bvp-1}-\eqref{eq:bvp-3} with $v^{\rm in}$ as the incident wave, i.e.,
	\begin{align}
		\label{eq:trans-1}
		\Delta^2 v^{\rm sc} -\kappa^4 v^{\rm sc} &=  0 \quad \text{ in } ~ D^{\rm c}, \\
		\label{eq:trans-2}
		( M (v^{\rm sc} + v^{\rm in}), N (v^{\rm sc} + v^{\rm in})) &= {\bm 0} \quad \text{ on } ~ \p D,\\
		\label{eq:trans-3}
		\lim\limits_{r=| \mbx |\to\infty}\sqrt{r}\left(\frac{\p v^{\rm sc}}{\p r}
		-\text{i}\kappa v^{\rm sc} \right)&=0.
	\end{align}
	By the superposition principle,  the far-field pattern $v^{\infty}$  of $v^{\rm sc}$ is indeed $F_D g$. Hence $F_D g =0 $ immediately yields $v^{\infty}=0$, which together with Rellich's lemma \cite[Lemma 4.1]{dong2024uniqueness} indicates that $v^{\rm sc}_{\rm pr}= -1/(2 \kappa^2)(\Delta v^{\rm sc}-\kappa^2 v^{\rm sc})$ vanishes in $D^{\rm c}$, i.e., the function $v^{\rm sc}$ satisfies the modified Helmholtz equation in $D^{\rm c}$ and decays exponentially as $| \mbx |\to \infty$. Besides, note the fact that $v^{\rm in}$ satisfies the Helmholtz equation with wavenumber $\kappa$ in $D$, which, combined with the boundary conditions \eqref{eq:trans-2}  of $v^{\rm sc}$ and $v^{\rm in}$, shows that $(v^{\rm in},v^{\rm sc})^\top$ satisfies the free transmission eigenvalue problem, i.e.,
	\begin{equation*}
		\begin{aligned}
			\Delta v^{\rm sc} - \kappa^2 v^{\rm sc}
			& = 0   \quad \mathrm{ in } ~ D^{\rm c}, \\
			\Delta v^{\rm in} + \kappa^2 v^{\rm in}
			& = 0   \quad \mathrm{ in } ~ D, \\
			(M (v^{\rm in} + v^{\rm sc}),
			N (v^{\rm in} +  v^{\rm sc}))
			& = \bm{0} \quad \mathrm{ on } ~ \p D.
		\end{aligned}
	\end{equation*}
	By the assumption of the wavenumber $\kappa$, we have that $v^{\rm in}$ vanishes in $D$ and thus on the whole $\bbR^2$ due to the analytic continuation principle. Hence $g$ is zero based on Theorem 3.27 in \cite{colton2019inverse}, which implies that the second statement of the theorem is proved.
	
	(c) We begin by asserting that the identity 
	\begin{equation}\label{eq:far-identity}
		F_D- F^{*}_D- \frac{\rm i}{4\pi} F^{*}_D F_D=0
	\end{equation}
	holds for Dirichlet and Neumann boundary conditions.
	For $g$, $h\in L^2(\bbS)$, let $v_g$ and $w_h$ be defined by
	\begin{equation*}
		v_g(\mbx)= \int_{ \bbS   } e^{{\rm i}\kappa \mbx\cdot  \mbd} g(\mbd)  \,\,{\rm d} s(\mbd) \quad\text{and}\quad
		w_h(\mbx)= \int_{ \bbS   } e^{{\rm i}\kappa \mbx\cdot  \mbd} h(\mbd)  \,\,{\rm d} s(\mbd) \quad
		\mbx\in \bbR^2.
	\end{equation*}
	Let $v$ and $w$ be the associated scattered fields of the Dirichlet or Neumann biharmonic scattering problem in $D^{\rm c}$. We use  the boundary conditions of $v+v_g$ and  $w+w_h$ to get
	\begin{align*}
		0={} &\int_{\p D}\Big\{ 
		[ (v+v_g)\, N (\overline{w+w_h})
		+ \p_{\mbn} (v+v_g) \,  M (\overline{w+w_h}) ] \\
		&\qquad\qquad   - [ N (v+v_g) \,    (\overline{w+w_h})
		+ M (v+v_g)\,   \p_{\mbn} (\overline{w+w_h})] 
		\Big\} \;\mathrm{d}s \\
		={} &\int_{\p D}\Big\{ 
		[ v\, N \overline{w}
		+ \p_{\mbn} v \,  M \overline{w} ]   - [ N v \,    \overline{w}
		+ M v\,   \p_{\mbn} \overline{w}] 
		\Big\} \;\mathrm{d}s \\
		&+ \int_{\p D}\Big\{ 
		[ v_g\, N \overline{w}
		+ \p_{\mbn} v_g \,  M \overline{w} ]   - [ N v_g \,    \overline{w}
		+ M v_g\,   \p_{\mbn} \overline{w}] 
		\Big\} \;\mathrm{d}s \\
		& + \int_{\p D}\Big\{ 
		[ v\, N \overline{w_h}
		+ \p_{\mbn} v \,  M \overline{w_h} ]    - [ N v \,    \overline{w_h}
		+ M v\,   \p_{\mbn} \overline{w_h}] 
		\Big\} \;\mathrm{d}s  \\
		& + \int_{\p D}\Big\{ 
		[ v_g\, N \overline{w_h}
		+ \p_{\mbn} v_g \,  M \overline{w_h} ]    - [ N v_g \,    \overline{w_h}
		+ M v_g\,   \p_{\mbn} \overline{w_h}] 
		\Big\} \;\mathrm{d}s 
		= I_1+ I_2 + I_3 +I_4.
	\end{align*}
	Let $R>0$ be sufficiently large such that $\overline{D}\subset B(\bm{0},R)$.
	It follows from  Green's formula in $B(\bm{0},R)\setminus\overline{D}$ that 
	$I_1$ is unchanged if the $\p D$ in the integral is replaced by $\p B(\bm{0},R)$.
	Note that from the Green's representation (see Lemma 2.2 in \cite{zhu2025direct}) of the radiating solution, we have
	the following asymptotic behavior
		\begin{align*}
			&[ v(\mbx)\, N_{\mbx}  \overline{w(\mbx)}
			+ \p_{\mbn(\mbx)} v(\mbx) \,  M_{\mbx}  \overline{w(\mbx)} ]   
			- [ N_{\mbx}  v(\mbx) \,    \overline{w(\mbx)}
			+ M_{\mbx} v(\mbx)\,   \p_{\mbn(\mbx)} \overline{w(\mbx)}] \\
			={}& - \frac{ {\rm i}\kappa^2}{2\pi|{\mbx}|}v^{\infty}(\hat{\mbx})\overline{w^{\infty}(\hat{\mbx})}
			+  O\Big(\frac{1}{|{\mbx}|^{3/2}}\Big).
	\end{align*}
	The representation of the far-field pattern (see Lemma 2.3 in \cite{zhu2025direct})
	shows that 
	\begin{equation*}
		I_2 = -2 \kappa^2(g,w^{\infty})_{L^2(\bbS)}\quad 
		\text{and}
		\quad I_3 = 2 \kappa^2(v^{\infty}, h)_{L^2(\bbS)}.
	\end{equation*}
	Using Green's formula in $D$, we have $I_4=0$. Letting $R$ tend to infinity, we prove the assertion.
	
	Define $\mathcal{L}=I + c F_D$, where $c=\frac{\rm i}{4\pi}$.
		By \eqref{eq:far-identity} we obtain that 
		\[\mathcal{L}^{*}\mathcal{L}=I + c(F_D -F_D^* - cF_D^* F_D) =I.\]
		As $\mathcal{L}$ is a Fredholm operator of index zero, it is unitary. Thus $F_D = (\mathcal{L} -I)/c$ is normal.
\end{proof}

\subsection{The case of Dirichlet obstacle}
In this subsection, we consider  the biharmonic scattering in the Dirichlet case. To this end, we  introduce the \emph{data-to-pattern operator} $\mcB_{{\rm Dir},D}\colon H^{3/2}(\p D)\times H^{1/2}(\p D)\to L^{2}(\bbS )$ such that for $(f_1,f_2)^\top\in H^{3/2}(\p D)\times H^{1/2}(\p D)$,
$
\mcB_{{\rm Dir},D}(f_1,f_2)^\top=v^{\infty},
$
where $v^{\infty}$ is the far-field pattern of $v$ and $v \in H_{\rm loc}^{2}(D^{\rm c})$ is the solution of the exterior Dirichlet boundary value problem  
\begin{align}
	\label{eq:gbvp-d-1}
	\Delta^2 v- \kappa^4v &= 0 \quad \mathrm{ in}\; D^{\rm c},\\
	\label{eq:gbvp-d-2}
	( v, \p_{\mbn} v)&= (f_1,f_2)  \quad \mathrm{on}\; \p D,\\
	\label{eq:gbvp-d-3}
	\lim\limits_{r\to\infty}\sqrt{r}\left(\frac{\p v}{\p r}
	-\mathrm{i}\kappa v \right)&=0,\quad r=| \mbx |.
\end{align}
Obviously, the operator $\mcB_{{\rm Dir},D}$ is well-defined according to Theorem 2.1 in \cite{bourgeois2020well}.
Then we investigate the properties of the operator  $\mcB_{{\rm Dir},D}$.

\begin{theorem}\label{thm:solution-operator}
	The operator $\mcB_{{\rm Dir},D}\colon H^{3/2}(\p D)\times H^{1/2}(\p D)\to L^{2}(\bbS)$ is compact with dense range.
\end{theorem}
\begin{proof}
	The compactness of $\mcB_{{\rm Dir},D}$ is analogous to the proof of Lemma 1.13 in \cite{kirsch2008factorization}. 
	Choose a sufficiently large $R>0$ such that $\overline{D}\subset B(\bm{ 0 },R)$. 
	Define the operators $\mcB_{D,1}\colon H^{3/2}(\p D)\times H^{1/2}(\p D)\to (C(\p B(\bm{ 0 },R) ))^4$ and $\mcB_{D,2}\colon(C(\p B(\bm{ 0 },R) ))^4 \to L^2(\bbS )$ by
	\begin{equation*}
		\mcB_{D,1}(f_1,f_2)^\top = (v|_{\p B(\bm{ 0 },R) },\p _{\mbn} v|_{\p B(\bm{ 0 },R) }, Mv|_{\p B(\bm{ 0 },R) }, Nv|_{\p B(\bm{ 0 },R) })^\top
	\end{equation*}
	and
	\begin{align*}
		\mcB_{D,2}{\bm g}(\hat{\mbx})=&
		-\frac{1}{2\kappa^{2}}
		\int_{\p  B(\bm{ 0 },R)}\Big\{
		[  e^{-\mathrm{i}\kappa\hat{\mbx}\cdot \bm{y}}g_4(\bm{y})
		+ (\p_{\mbn(\bm{y})}e^{-\mathrm{i}\kappa\hat{\mbx}\cdot \bm{y}})  g_3(\bm{y}) ] \\
		& \qquad\qquad\qquad   - [ (N_{\bm{y}}e^{-\mathrm{i}\kappa\hat{\mbx}\cdot \bm{y}}) g_1(\bm{y})
		+ (M_{\bm{y}}e^{-\mathrm{i}\kappa\hat{\mbx}\cdot \bm{y}})  g_2(\bm{y}) ]
		\Big\} \;\mathrm{d}s(\bm{y}) \quad \hat{\mbx}\in \bbS
	\end{align*}
	respectively. It is obvious that $\mcB_{{\rm Dir},D}= \mcB_{D,2} \mcB_{D,1}$ holds. By the interior regularity of the elliptic equation, the operator $\mcB_{D,1}$ is compact and thus  $\mcB_{{\rm Dir},D}$ is compact.
	
	Next, we prove the density of the range of $\mcB_{{\rm Dir},D}$ in $L^2(\bbS )$. 
	We claim that  $\mcB_{{\rm Dir},D}^{*}$, the adjoint of  $\mcB_{{\rm Dir},D}$, is given by
	\begin{equation}
		\label{eq:op-b-adjoint}
		\mcB_{{\rm Dir},D}^{*} g = \frac{1}{2\kappa^{2}}
		\begin{pmatrix}
			N \overline{U},\\
			M \overline{U}
		\end{pmatrix}\Big|_{\p D},
	\end{equation}
	where  $U$ is defined by $U=U^{\rm in}+U^{\rm sc}$ with  $U^{\rm in}(\mbx)=\int_{ \bbS   } e^{-{\rm i}\kappa \mbx \cdot \mbd}\overline{g(\mbd)} \,\,{\rm d} s(\mbd) \quad \mbx\in \bbR^2$,  and $U^{\rm sc}$ being the solution to \eqref{eq:bvp-1}-\eqref{eq:bvp-3} subject to $(f_1,f_2)=-(U^{\rm in}|_{\p D},\p _{\mbn}U^{\rm in}|_{\p D})$.
	For any $(f_1,f_2)^\top \in H^{3/2}(\p D)\times H^{1/2}(\p D)$, suppose that the operator $\mcB_{{\rm Dir},D}$ maps $(f_1,f_2)^\top$ to the far-field pattern $v^{\infty}$, where  $v$ is the radiating solution of \eqref{eq:bvp-1} with $(f_1,f_2)^\top$ as the  boundary data.
	From the representation of the far-field pattern, i.e., Lemma 2.2 in \cite{zhu2025direct}, we obtain
	\begin{align*}
		v^{\infty}(\hat{\mbx}) = {}& -\frac{1}{2\kappa^{2}}
		\int_{\p D}\Big\{
		[  e^{-\mathrm{i}\kappa\hat{\mbx}\cdot \bm{y}}N v(\bm{y})
		+ (\p_{\mbn(\bm{y})}e^{-\mathrm{i}\kappa\hat{\mbx}\cdot \bm{y}})  M v(\bm{y}) ] \\
		& \qquad\qquad\quad   - [ (N_{\bm{y}}e^{-\mathrm{i}\kappa\hat{\mbx}\cdot \bm{y}}) v(\bm{y})
		+ (M_{\bm{y}}e^{-\mathrm{i}\kappa\hat{\mbx}\cdot \bm{y}})  \p_{\mbn} v(\bm{y}) ]
		\Big\} \;\mathrm{d}s(\bm{y}) \quad \hat{\mbx}\in \bbS.
	\end{align*}
	This, the Dirichlet boundary conditions of both $U$ and $v$,  and changing the order of integration yield
	\begin{align*}
		&(\mcB_{{\rm Dir},D}(f_1, f_2)^\top,g )_{L^2(\bbS)}
		=  \int_{ \bbS   } v^{\infty}(\hat{\mbx}) \overline{g(\hat{\mbx} )}
		\,\,{\rm d} s(\hat{\mbx} )  \\
		={} & -\frac{1}{2\kappa^{2}}  \int_{\p D}\Big\{
		[ U^{\rm in}(\bm{ y })\, N v(\bm{y})
		+ \p_{\mbn(\bm{y})} U^{\rm in}(\bm{ y })\,  M v(\bm{y}) ] \\
		& \qquad\qquad\quad   - [ N_{\bm{y}} U^{\rm in}(\bm{ y }) \,   v(\bm{y})
		+ M_{\bm{y}} U^{\rm in}(\bm{ y })\,   \p_{\mbn} v(\bm{y}) ]
		\Big\} \;\mathrm{d}s(\bm{y}) \\
		={} & -\frac{1}{2\kappa^{2}}  \int_{\p D}\Big\{
		[( U^{\rm in}(\bm{ y }) + U^{\rm sc}(\bm{ y }) )\,  N v(\bm{y})
		+ \p_{\mbn(\bm{y})} (U^{\rm in}(\bm{ y })  + U^{\rm sc}(\bm{ y }) ) \,  M v(\bm{y}) ] \\
		& \qquad\qquad\quad   - [ N_{\bm{y}} (U^{\rm in}(\bm{ y })  + U^{\rm sc}(\bm{ y }) )     \, v(\bm{y})
		+ M_{\bm{y}}(U^{\rm in}(\bm{ y }) + U^{\rm sc}(\bm{ y }) ) \,   \p_{\mbn} v(\bm{y}) ]
		\Big\} \;\mathrm{d}s(\bm{y}) \\
		={} &  \left \langle\begin{pmatrix}f_1\\f_2 \end{pmatrix} ,
		\frac{1}{2\kappa^{2}}  \begin{pmatrix}
			N \overline{U},\\
			M \overline{U}
		\end{pmatrix}\right\rangle
		\coloneqq \left \langle\begin{pmatrix}f_1\\f_2 \end{pmatrix} ,  \mcB_{{\rm Dir},D}^{*} g\right \rangle,
	\end{align*}
	where $\langle \cdot ,\cdot \rangle$ is the dual product between $H^{3/2}\times H^{1/2}$ and $H^{-3/2}\times H^{-1/2}$, and  we use the following fact that
	\begin{align*}
		\int_{\p D}\Big\{ &
		[ U^{\rm sc}(\bm{ y })\, N v(\bm{y})
		+ \p_{\mbn(\bm{y})} U^{\rm sc}(\bm{ y })\,  M v(\bm{y}) ] \\
		&   - [ N_{\bm{y}} U^{\rm sc}(\bm{ y }) \,    v(\bm{y})
		+ M_{\bm{y}} U^{\rm sc}(\bm{ y })\,   \p_{\mbn} v(\bm{y}) ]
		\Big\} \;\mathrm{d}s(\bm{y}) =0
	\end{align*}
	because both $U^{\rm sc}$ and $v$ satisfy the Sommerfeld radiation condition. Then we arrive at \eqref{eq:op-b-adjoint}.
	
	To prove the denseness of ${\rm Ran} (\mcB_{{\rm Dir},D})$, it suffices to prove the injectivity of $\mcB_{{\rm Dir},D}^{*}$. To this end,
	let $\mcB_{{\rm Dir},D}
	^{*}g={\bm 0}$ for some $g \in L^{2}(\bbS )$. From \eqref{eq:op-b-adjoint}, it is easy to see that $U|_{\p D}=\p_{\mbn}U|_{\p D}=MU|_{\p D}=NU|_{\p D}=0$, which together with the Holmgren's uniqueness theorem shows that $U$ vanishes in $\bbR^2\setminus\overline{D}$. Thus $U^{\rm in}$ satisfies the Sommerfeld radiation condition. Since  $U^{\rm in}$ satisfies  the Helmholtz equation in $\bbR^2$,  $U^{\rm in}$ vanishes in $\bbR^2$. According to \cite[Theorem 3.27]{colton2019inverse}, we have $g$ vanishes, which completes the proof.
\end{proof}

To relate the far-field operator $F_D$ and the data-to-pattern operator $\mcB_{{\rm Dir},D}$, we introduce the  \emph{data-to-obstacle operator} $\mathcal{H}_{{\rm Dir},D}\colon L^2(\bbS )\to H^{3/2}(\p D)\times H^{1/2}(\p D)$ such that for $\mbx\in \p D$
\begin{equation}
	\label{eq:op-incident}
	\mathcal{H}_{{\rm Dir},D} g(\mbx ) \coloneqq
	\begin{pmatrix}    \displaystyle
		\int_{ \bbS   } e^{{\rm i}\kappa \mbx\cdot \mbd} g(\mbd) \,\,{\rm d} s(\mbd)\\[5mm]
		\displaystyle
		\frac{\p }{ \p \mbn (\mbx) }\int_{ \bbS   } e^{{\rm i}\kappa \mbx\cdot \mbd} g(\mbd) \,\,{\rm d} s(\mbd)
	\end{pmatrix}.
\end{equation}
By virtue of the definitions of $F_D$, $\mcB_{{\rm Dir},D}$, and $\mathcal{H}_{{\rm Dir},D}$, the superposition principle implies $F_D=-\mcB_{{\rm Dir},D} \mathcal{H}_{{\rm Dir},D}$. To get a further factorization of $F_D$, we introduce the mapping $\mcS_{{\rm Dir},D}\colon H^{-3/2}(\p D)\times H^{-1/2}(\p D)\to H^{3/2}(\p D)\times H^{1/2}(\p D)$ by
\begin{equation}
	\label{eq:op-s}
	\displaystyle
	\mcS_{{\rm Dir},D}\begin{pmatrix} \tau \\ \sigma \end{pmatrix} ({\mbx})
	\coloneqq \begin{pmatrix}    \displaystyle
		\int_{ \p D  } \left( G(\mbx,\bm{ y }) \tau(\bm{ y }) + \dfrac{\p G(\mbx,\bm{ y })}{\p \mbn(\bm{ y})} \sigma(\bm{ y })\right) \,\,{\rm d}  s(\bm{ y }) \\[5mm]
		\displaystyle
		\dfrac{\p }{\p \mbn(\mbx)}
		\int_{ \p D  } \left( G(\mbx,\bm{ y }) \tau(\bm{ y }) + \dfrac{\p G(\mbx,\bm{ y })}{\p \mbn(\bm{ y})} \sigma(\bm{ y })\right) \,\,{\rm d} s(\bm{ y })
	\end{pmatrix},
\end{equation}
for $\mbx\in  \p D$, which is a bounded operator in the light of the potential theory of the biharmonic equation.

Now we are in a position to get the symmetric factorization of far-field operator.
\begin{theorem}[The symmetric factorization of $F_D$]
	The factorization
	\begin{equation}\label{eq:far-field-factorization}
		F_D= 2 \kappa^2 \mcB_{{\rm Dir},D} \mcS_{{\rm Dir},D}^{*} \mcB_{{\rm Dir},D}^{*}
	\end{equation}
	holds.
\end{theorem}
\begin{proof}
	It is clear that $\mathcal{H}_{{\rm Dir},D}^{*}\colon H^{-3/2}(\p D)\times H^{-1/2}(\p D) \to L^2(\bbS)$ is given by
	\begin{equation*}
		\mathcal{H}_{{\rm Dir},D}^{*} \begin{pmatrix} \tau\\
			\sigma \end{pmatrix} (\hat{\mbx}) 
		= \int_{ \p D  } \left( e^{-{\rm i}\kappa \hat{\mbx}\cdot \bm{ y }}  \tau(\bm{ y }) + \dfrac{\p e^{-{\rm i}\kappa \hat{\mbx}\cdot \bm{ y }} }{\p \mbn(\bm{ y})} \sigma(\bm{ y })\right) \,\,{\rm d}  s(\bm{ y })   \quad \hat{\mbx} \in \bbS .
	\end{equation*}
	Define the function
	\begin{equation*}
		v({\mbx}) = -2 \kappa^2\int_{ \p D  } \left( G(\mbx,\bm{ y }) \tau(\bm{ y }) + \dfrac{\p G(\mbx,\bm{ y })}{\p \mbn(\bm{ y})} \sigma(\bm{ y })\right) \,\,{\rm d}  s(\bm{ y }) \quad {\mbx} \in  D^{\rm c}.
	\end{equation*}
	It follows from the asymptotic behavior  of $G$ that the far-field pattern $v^{\infty}$ of $v$ is exactly $\mathcal{H}_{{\rm Dir},D}^{*} ( \tau, \sigma )^\top$. It is easy to see that
	\begin{equation*}
		\mathcal{H}_{{\rm Dir},D}^{*} ( \tau, \sigma )^\top = \mcB_{{\rm Dir},D} (v|_{\p D},\p _{\mbn}v|_{\p D})^\top= -2 \kappa^2 \mcB_{{\rm Dir},D} \mcS_{{\rm Dir},D} ( \tau, \sigma )^\top.
	\end{equation*}
	Thus, $\mathcal{H}_{{\rm Dir},D}^{*}=-2 \kappa^2 \mcB_{{\rm Dir},D} \mcS_{{\rm Dir},D}$ 
	and $\mathcal{H}_{{\rm Dir},D}=-2 \kappa^2\mcS_{{\rm Dir},D}^{*} \mcB_{{\rm Dir},D}^{*}$, which leads to the desired result by $F_D=-\mcB_{{\rm Dir},D}\mathcal{H}_{{\rm Dir},D}$. 
\end{proof}

To explore  deeper properties of $\mcB_{{\rm Dir},D}$, we introduce a Dirichlet-to-Neumann operator $\Lambda_e \colon H^{3/2}(\p D)\to H^{1/2}(\p D) $, which maps $f \in H^{3/2}(\p D)$ into the Neumann trace $\p _{\mbn}v$, where $v \in H^2(D^{\rm c})$ is the unique solution to exterior Dirichlet boundary value problem
\begin{align}
	\label{eq:yukawa}
	\Delta v -\kappa^2 v  &=  0 \quad \mathrm{ in } ~ D^{\rm c}, \\
	\label{eq:yukawa-b}
	v &=  f  \quad \mathrm{ on } ~ \p D.
\end{align}
It is clear that the operator $\Lambda_e$ is bounded from $H^{3/2}(\p D)$ to $H^{1/2}(\p D)$ by potential theory of the modified Helmholtz equation (or the Yukawa equation). Then the kernel of the operator $\mcB_{{\rm Dir},D}$ can be characterized as follows.

\begin{theorem}\label{thm:kernel-B}
	The kernel   ${\rm Ker}(\mcB_{{\rm Dir},D})$  of the operator $\mcB_{{\rm Dir},D}$ is given by
	\begin{equation*}
		{\rm Ker}(\mcB_{{\rm Dir},D}) =\left\{ (f, \Lambda_e f)^\top
		\colon
		f \in  H^{3/2}(\p D)
		\right\}.
	\end{equation*}
\end{theorem}
\begin{proof}
	Firstly, we shall show ${\rm Ker}(\mcB_{{\rm Dir},D})\subset  \left\{
	(f, \Lambda_e f)^\top \colon f \in  H^{3/2}(\p D) \right\} $.
	Given any $(f_1,f_2)^\top \in {\rm Ker}(\mcB_{{\rm Dir},D})$, $v^{\infty}=\mcB_{{\rm Dir},D}(f_1,f_2)^\top =0$ if we assume that $v$ is the radiating solution of the biharmonic wave equation in $D^{\rm c}$ with $(f_1,f_2)^\top$ as the boundary data. Then Rellich's lemma  \cite[Lemma 4.1]{dong2024uniqueness} implies $v_{\rm pr}=0$, where $v_{\rm pr}= -1/(2 \kappa^2)(\Delta v - \kappa^2 v)$. From this we know that $\Delta v -\kappa^2 v =0$ holds in $D^{\rm c}$, which, combined with $(f_1,f_2)^\top=(v|_{\p D},\p _{\bm{ n}}v|_{\p D})^\top$, yields $f_2=\Lambda_e f_1$.
	
	Next, we show $ \left\{
	(f, \Lambda_e f)^\top \colon f \in  H^{3/2}(\p D) \right\}\subset  {\rm Ker}(\mcB_{{\rm Dir},D})$. Let $f \in  H^{3/2}(\p D)$ and there exists a function $v \in H^2(D^{\rm c})$ satisfying \eqref{eq:yukawa}-\eqref{eq:yukawa-b}  such that $\Lambda_e f=\p _{\mbn}v |_{\p D}$. Besides, we notice that the function $v$ decays exponentially as $| \mbx | \to \infty$. Hence, the function $v$ is the solution of the following exterior boundary value problem:
	\begin{align*}
		\Delta^2 v -\kappa^4 v &=  0 \quad \text{ in } ~ D^{\rm c}, \\
		(v, \p _{\mbn} v)&= (f,  \Lambda_e f) \quad \text{ on } ~ \p D,\\
		\lim\limits_{r=| \mbx |\to\infty}\sqrt{r}\left(\frac{\p v}{\p r}
		-\text{i}\kappa v \right)&=0.
	\end{align*}
	Consequently, in view of the exponential decay of the function $v$, the far-field pattern $v^{\infty}$ of $v$ is identically zero, i.e., $\mcB_{{\rm Dir},D}(f,\Lambda_e f)^\top=0$.   Therefore, we conclude that the proof is complete.
\end{proof}

We introduce $\tilde{\mcS}_{{\rm Dir},D} \colon H^{-3/2}(\p D)\times H^{-1/2}(\p D)\to H^{3/2}(\p D)\times H^{1/2}(\p D)$ by
\begin{equation*}
	\displaystyle
	\tilde{\mcS}_{{\rm Dir},D}\begin{pmatrix} \tau \\ \sigma \end{pmatrix}({\mbx})
	\coloneqq \begin{pmatrix}    \displaystyle
		\int_{ \p D  } \left( \tilde{G}(\mbx,\bm{ y }) \tau(\bm{ y }) + \dfrac{\p \tilde{G}(\mbx,\bm{ y })}{\p \mbn(\bm{ y})} \sigma(\bm{ y })\right) \,\,{\rm d}  s(\bm{ y }) \\[5mm]
		\displaystyle
		\dfrac{\p }{\p \mbn(\mbx)}
		\int_{ \p D  } \left( \tilde{G}(\mbx,\bm{ y }) \tau(\bm{ y }) + \dfrac{\p \tilde{G}(\mbx,\bm{ y })}{\p \mbn(\bm{ y})} \sigma(\bm{ y })\right) \,\,{\rm d} s(\bm{ y })
	\end{pmatrix}
	\, \mbx\in  \p D,
\end{equation*}
with $\tilde{G}(\mbx,\bm{ y })=1/8 [H^{(1)}_0( e^{{\rm i}3\pi/4}|\mbx-\bm{ y }|) - H^{(1)}_0( e^{{\rm i}\pi/4}|\mbx-\bm{ y }|)]$, which is the fundamental solution to the equation $(\Delta^2 + 1) v=0$.

The following theorem describes some properties of $\mcS_{{\rm Dir},D}$ and $\tilde{\mcS}_{{\rm Dir},D}$.
\begin{theorem}\label{thm:middle}
	Let $\mathcal{X}:= H^{-3/2}(\p D)\times H^{-1/2}(\p D)$. Then we have
	\begin{enumerate}
		\item[(a)] The operator $\mcS_{{\rm Dir},D}$ is an isomorphism from $\mathcal{X}$ to $\mathcal{X}^*$ if $\kappa$ is not a fourth root of a Dirichlet eigenvalue of the bi-Laplacian operator $\Delta^2$ in $D$.
		\item[(b)]  $\mathrm{Im}\, \left\langle
		\begin{pmatrix} \tau\\ \sigma\end{pmatrix},
		\mcS_{{\rm Dir},D} \begin{pmatrix} \tau\\ \sigma \end{pmatrix}
		\right\rangle >0$  for all   $\begin{pmatrix} \tau\\ \sigma\end{pmatrix} \in \overline{{\rm Ran}(\mcB_{{\rm Dir},D}^{*})}$ with $\begin{pmatrix} \tau\\ \sigma \end{pmatrix}\neq \bm{ 0 }$ if $\kappa$ is not a clamped transmission eigenvalue.
		
		\item[(c)] $- \tilde{\mcS}_{{\rm Dir},D}\colon \mathcal{X}\to \mathcal{X}^*$ is self-adjoint and  positive coercive.
		
		\item[(d)]  $\mcS_{{\rm Dir},D} - \tilde{\mcS}_{{\rm Dir},D}\colon \mathcal{X}\to \mathcal{X}^*$ is compact.
	\end{enumerate}
\end{theorem}
\begin{proof}
	(a) The isomorphism of $\mcS_{{\rm Dir},D}$ follows directly from Proposition 2 in \cite{bourgeois2020linear}.
	
	(b) Let $(\tau,\sigma)^\top \in \mathcal{X}$ and we prove
	\begin{equation*}
		\mathrm{Im}\, \left\langle
		\begin{pmatrix}
			\tau\\ \sigma
		\end{pmatrix}
		, \mcS_{{\rm Dir},D}\begin{pmatrix}
			\tau\\ \sigma
		\end{pmatrix}\right\rangle \geq 0.
	\end{equation*}
	To see this, define the following potential of the biharmonic wave equation
	\begin{equation*}
		v(\mbx)= \int_{ \p D  } \left[ G(\mbx, \bm{ y }) \tau(\bm{ y }) +\dfrac{\p G(\mbx,\bm{ y })}{\p \mbn(\bm{ y })}\sigma(\bm{ y })
		\right]  \,\,{\rm d}s(\bm{ y }) \quad \mbx\not \in \p D.
	\end{equation*}
	Then $v \in H^{2}_{\mathrm{ loc } }(\bbR^2)$ satisfies the biharmonic wave equation in $\bbR^2\setminus \p D$. The jump relations of the single- and double-layer potentials imply
	\begin{align}
		\label{eq:jump-1}
		0 & = v|_{+} - v|_{-}\quad
		0 =\p_{\mbn}  v|_{+}-  \p_{\mbn} v|_{-}    \quad \mathrm{ on } ~ \p D, \\
		\label{eq:jump-2}
		\tau & = N\, v|_{+} - N\, v|_{-}\quad
		\sigma = M\, v|_{+} - M\, v|_{-}  \quad \mathrm{ on } ~ \p D.
	\end{align}
	Using the integration by parts formula in $D$ and $D_R\coloneqq\{ \mbx \in D^{\rm c}: | \mbx | < R\}$, we have
	\begin{align*}
		&\left\langle
		\begin{pmatrix}
			\tau\\ \sigma
		\end{pmatrix}
		, \mcS_{{\rm Dir},D}\begin{pmatrix}
			\tau\\ \sigma
		\end{pmatrix}\right\rangle
		=   \left\langle
		\begin{pmatrix}
			N v|_{+} - N v|_{-}\\
			Mv|_{+} - Mv|_{-} 
		\end{pmatrix},
		\begin{pmatrix}
			v\\ \p _{\mbn} v
		\end{pmatrix}   \right\rangle\\
		={}&  - \int_{ D\cup D_R  } \left( a(v,\overline{v}) -\kappa^4 | v |^2\right) \,\,{\rm d} \mbx
		+  \int_{ | \mbx | = R  } [\overline{v} \,Nv+ \p _{\mbn}\overline{v}\, Mv   ] \,\,{\rm d} s(\mbx).
	\end{align*}
	Taking the imaginary part, we can get
	\begin{align} \notag
		\mathrm{Im}\, \left\langle
		\begin{pmatrix}
			\tau\\ \sigma
		\end{pmatrix}
		, \mcS_{{\rm Dir},D}\begin{pmatrix}
			\tau\\ \sigma
		\end{pmatrix}\right\rangle 
		&=  \mathrm{Im}\,  \int_{ | \mbx | = R  }
		[\overline{v}\, Nv+ \p _{\mbn}\overline{v}\, Mv   ]
		\,\,{\rm d} s(\mbx)\\ 
		\label{eq:single-imaginary}
		&=  \frac{\kappa^{2}}{4\pi}\int_{\bbS}|v^{\infty}|^2\mathrm{d}s \geq 0,
	\end{align}
	where we use Lemma 2.3 in \cite{zhu2025direct}.
	
	We now prove that \[
	\mathrm{Im}\, \left\langle
	\begin{pmatrix}
		\tau\\ \sigma
	\end{pmatrix}
	, \mcS_{{\rm Dir},D}\begin{pmatrix}
		\tau\\ \sigma
	\end{pmatrix}\right\rangle = 0 \quad \text{for some} \quad \begin{pmatrix}
		\tau\\ \sigma
	\end{pmatrix} \in   \overline{{\rm Ran}(\mcB_{{\rm Dir},D}^{*})}\]
	implies  that $( \tau, \sigma)^\top= \bm{ 0 }$. To prove this, we introduce the annihilator set  
	\begin{align*}
			\prescript{a}{}{[{\rm Ker}(\mcB_{{\rm Dir},D})]}
			\coloneqq {} \Bigg\{ \begin{pmatrix}
				\tau\\ \sigma
			\end{pmatrix} \in   \mathcal{X} :  
			\left \langle  \begin{pmatrix}
				\tau\\ \sigma
			\end{pmatrix},  \begin{pmatrix}
				h\\ t
			\end{pmatrix}\right\rangle=0  \,\, \text{for all} \,\,
			\begin{pmatrix}
				h\\ t
			\end{pmatrix}\in {\rm Ker}(\mcB_{{\rm Dir},D})\Bigg\}
	\end{align*}
	of  ${\rm Ker}(\mcB_{{\rm Dir},D})$ and  employ the space relation
	$ \overline{{\rm Ran}(\mcB_{{\rm Dir},D}^{*})} = \prescript{a}{}{[{\rm Ker}(\mcB_{{\rm Dir},D})]}$.
	From \eqref{eq:single-imaginary} and Rellich's lemma  \cite[Lemma 4.1]{dong2024uniqueness}, we have that
	$v_{\rm pr}=0$ and $v=v_{\rm ev}$ in $D^{\rm c}$, where $v_{\rm pr}=-1/(2 \kappa^2)(\Delta v -\kappa^2 v)$ and $v_{\rm ev}=1/(2 \kappa^2)(\Delta v  + \kappa^2 v)$.  If $\mbz\in D$, it is clear that $ \overline{\Phi}_{{\rm i}\kappa}(\cdot ,\mbz)|_{\p D} \in H^{3/2}(\p D)$
	and $\overline{\Phi}_{{\rm i}\kappa}(\cdot ,\mbz)$ satisfies the modified Helmholtz equation \eqref{eq:yukawa} in $D^{\rm c}$. By
	Theorem \ref{thm:kernel-B}, 
	$(\overline{\Phi}_{{\rm i}\kappa}(\cdot ,\mbz),  \p _{\mbn(\cdot )} \overline{\Phi}_{{\rm i}\kappa}(\cdot ,\mbz))^\top \in {\rm Ker}(\mcB_{{\rm Dir},D})$. Owing to $  (\tau, \sigma)^\top \in  \overline{{\rm Ran}(\mcB_{{\rm Dir},D}^{*})}=\prescript{a}{}{[{\rm Ker}(\mcB_{{\rm Dir},D})]} $, we can obtain that
	\begin{equation*}
		\left \langle  \begin{pmatrix}
			\tau\\ \sigma
		\end{pmatrix},  \begin{pmatrix}
			\overline{\Phi}_{{\rm i}\kappa}(\cdot ,\mbz)\\  \p _{\mbn(\cdot )} \overline{\Phi}_{{\rm i}\kappa}(\cdot ,\mbz)
		\end{pmatrix}\right\rangle=0,
	\end{equation*}
	i.e.,
	\begin{equation*}
		\int_{ \p D  }\big[ \tau(\bm{ y }) \Phi_{{\rm i}\kappa}(\bm{ y } ,\mbz)
		+  \sigma(\bm{ y })  \p _{\mbn(\bm{ y } )}\Phi_{{\rm i}\kappa}(\bm{ y } ,\mbz) \big]
		\,\,{\rm d}s(\bm{ y }) =0 \quad \text{for all} \quad \mbz \in  D.
	\end{equation*}
	Then $v_{\rm ev}=0$  and $v=v_{\rm pr}$ in $D$. Choosing $w_1=v|_{D^{\rm c}}=v_{\rm ev}|_{D^{\rm c}}$  and  $w_2=-v|_D=-v_{\rm pr}|_D$, we deduce from the jump relations \eqref{eq:jump-1} that $(w_1,w_2)^\top$ satisfies the clamped transmission eigenvalue problem \eqref{eq:clamped}:
	\begin{equation*}
		\begin{aligned}
			\Delta w_1 - \kappa^2 w_1 & = 0   \quad \mathrm{ in } ~ D^{\rm c}, \\
			\Delta w_2 + \kappa^2 w_2 & = 0   \quad \mathrm{ in } ~ D, \\
			(w_1 + w_2,\p _{\mbn} w_1 + \p _{\mbn} w_2)   & = {\bm 0} \quad \mathrm{ on } ~ \p D.
		\end{aligned}
	\end{equation*}
	Since $\kappa$ is not a clamped transmission eigenvalue,  we conclude that $(w_1,w_2)^\top=\bm{ 0 }$ and thus $v\equiv 0$ in $\bbR^2$. Consequently, the jump relations \eqref{eq:jump-2} yields $(\tau,\sigma)^\top=\bm{ 0 }$.
	
	(c) By the following formula  for the Hankle functions (see Eq. (10.11.5) in \cite{olverNISTHandbookMathematical2010})
		\[H^{(1)}_0(z e^{\pi \rm i})=-H^{(2)}_0(z) \quad z\not= 0,\, z\in \mathbb{C},\]
		we have $H^{(1)}_0( e^{{3\pi \rm i}/4} r) = -H^{(2)}_0(e^{{-\pi \rm i}/4} r)$ for $r>0$. Note that 
		$H^{(1)}_0(\overline{z}) = \overline{H^{(2)}_0(z)}$ holds for $z\not= 0$ and $z\in \mathbb{C}$ (see Eq. (10.11.9) in \cite{olverNISTHandbookMathematical2010})), which indicates that $H^{(1)}_0( e^{{3\pi \rm i}/4} r)=  -\overline{H^{(1)}_0(e^{{\pi \rm i}/4} r)}$ for $r>0$ and thus $\tilde{G}(\mbx,\bm{ y })$ takes real values.
	Let
	\[\mathbb{G}_{\rm Dir}(\mbx,\bm{ y })\coloneqq 
	\begin{pmatrix}
		\tilde{G}(\mbx,\bm{ y })
		&   \displaystyle  \frac{\p \tilde{G}(\mbx,\bm{ y })}{\p \mbn(\bm{ y })}\\[3mm]
		\displaystyle
		\frac{\p \tilde{G}(\mbx,\bm{ y })}{\p \mbn(\mbx)}
		&   \displaystyle \frac{\p^2 \tilde{G}(\mbx,\bm{ y })}{\p \mbn(\mbx) \p \mbn(\bm{ y })}
	\end{pmatrix}.
	\]
	Then the self-adjointness of $\tilde{\mcS}_{{\rm Dir},D}$ follows from the facts that  the function $\tilde{G}(\mbx,\bm{ y })$ is real-valued and the relation $\mathbb{G}_{\rm Dir}(\mbx,\bm{ y })=  \mathbb{G}_{\rm Dir}(\bm{ y },\mbx)^\top$ holds.
	
	Now we turn to proving the coercivity of  $-\tilde{\mcS}_{{\rm Dir},D}$.
	By an argument similar to that used in Proposition 2 of
	\cite{bourgeois2020linear} to establish the isomorphism of the operator $\mcS_{\mathrm{Dir},D}$, we obtain that the operator $\tilde{\mcS}_{{\rm Dir},D}$ is an isomorphism from $\mathcal{X}$ to $\mathcal{X}^*$. Given $(\tau,\sigma)^\top\in \mathcal{X}$,
		consider the function
	\begin{equation*}
		v(\mbx)= \int_{ \p D  } \left[ \tilde{G}(\mbx, \bm{ y }) \tau(\bm{ y }) +\dfrac{\p \tilde{G}(\mbx,\bm{ y })}{\p \mbn(\bm{ y })}\sigma(\bm{ y })
		\right]  \,\,{\rm d}s(\bm{ y }) \quad \mbx\not \in \p D.
	\end{equation*}
	We observe that the function $v$ decays exponentially as $| \mbx | \to \infty$. With this,
	it follows from the Green's formulas in  $D$ and $D_R$ that
	\begin{align*}
		- \left\langle
		\begin{pmatrix}
			\tau\\ \sigma
		\end{pmatrix}
		, \tilde{\mcS}_{{\rm Dir},D}\begin{pmatrix}
			\tau\\ \sigma
		\end{pmatrix}\right\rangle
		&=  \int_{ D\cup D^{\rm c}  } \left( a(v,\overline{v}) +  | v |^2\right) \,\,{\rm d} \mbx
		\geq c \| v \|_{H^2(\bbR^2)}^2,
	\end{align*}
	where we let $R$ tend to infinity and $c>0$ is a constant. By trace theorem (see Lemma 5.7.1 in \cite{Hsiao2021boundary}) and the boundedness of $\tilde{\mcS}_{{\rm Dir},D}^{-1}$, we deduce that
		\begin{align*}
			- \left\langle
			\begin{pmatrix}
				\tau\\ \sigma
			\end{pmatrix}
			, \tilde{\mcS}_{{\rm Dir},D}\begin{pmatrix}
				\tau\\ \sigma
			\end{pmatrix}\right\rangle
			& \geq c \| v \|_{H^2(\bbR^2)}^2
			\geq c \Big\|\begin{pmatrix}
				v\\ \p _{\mbn} v
			\end{pmatrix} \Big\|_{\mathcal{X}^*}^2\\
			&  \geq c \Big\|  \tilde{\mcS}_{{\rm Dir},D}\begin{pmatrix}
				\tau\\ \sigma
			\end{pmatrix}\Big\|_{\mathcal{X}^*}^2 
			\geq c \Big\|\begin{pmatrix}
				\tau\\ \sigma
			\end{pmatrix}\Big\|_{\mathcal{X}}^2,
	\end{align*}
	where $c>0$ is a constant.
	
	(d) The compactness of $\mcS_{{\rm Dir},D}-\tilde{\mcS}_{{\rm Dir},D}$ follows from the fact that the kernel $G-\tilde{G}$ is of the form 
		\begin{align*}
			G(\mbx,{\bm y})- \tilde{G}(\mbx,{\bm y})
			=   A_1(| \mbx-\bm{ y } | ) | \mbx-\bm{ y } |^4 \ln (| \mbx-\bm{ y } |) + A_2(| \mbx-\bm{ y } |^2),
	\end{align*}
	where $A_1$ and $A_2$ are analytic functions.
\end{proof}

Apply Theorem \ref{thm:range-1}  to the factorization \eqref{eq:far-field-factorization} with $H=L^2(\bbS)$, $X=H^{3/2}(\p D)\times H^{1/2}(\p D)$, $F=F_D$, $B=\mcB_{{\rm Dir},D}$ and $A=2 \kappa^2\mcS_{{\rm Dir},D}^{*}$.
The assumptions on $F=F_D$, $A=2 \kappa^2 \mcS_{{\rm Dir},D}^{*}$ and $A_0=2 \kappa^2\tilde{\mcS}_{{\rm Dir},D}^{*}$
are verified in Theorems \ref{thm:far-field} and \ref{thm:middle} respectively. In addition, apply Theorem \ref{thm:range-2} to the factorization \eqref{eq:far-field-factorization} with $Y=L^2(\bbS)$, $X=H^{3/2}(\p D)\times H^{1/2}(\p D)$, $F=-F_D$, $G=\mcB_{{\rm Dir},D}$, $T=-2 \kappa^2\mcS_{{\rm Dir},D}^{*}$, $T_0=-2 \kappa^2\tilde{\mcS}_{{\rm Dir},D}^{*}$ and $t=0$. The  assumptions on  $G=\mcB_{{\rm Dir},D}$, $T=-2 \kappa^2\mcS_{{\rm Dir},D}^{*}$ and $T_0=-2 \kappa^2\tilde{\mcS}_{{\rm Dir},D}^{*}$ are satisfied in Theorems \ref{thm:solution-operator} and \ref{thm:middle} respectively. Thus we have the following range identities.

\begin{theorem}\label{thm:range-identities}
	Assume that $\kappa$ is not a clamped transmission eigenvalue. Then
	\begin{equation*}
		{\rm Ran}((F_D^{*}F_D)^{1/4} )={\rm Ran}(F_{D,\#}^{1/2})={\rm Ran}(\mcB_{{\rm Dir},D}).
	\end{equation*}
\end{theorem}

Now we are ready to give another characterization of the range $ {\rm Ran}(\mcB_{{\rm Dir},D})$ of the operator $\mcB_{{\rm Dir},D}$.
\begin{theorem}\label{thm:solution-range}
	For $\mbz\in \bbR^2$, set $\phi_{\mbz}(\bm{ \hat{x}  })\coloneqq e^{-{\rm i}\kappa \hat{\mbx} \cdot \mbz}$ $\hat{\mbx} \in \bbS $. Then,
	\[
	\phi_{\mbz} \in {\rm Ran}(\mcB_{{\rm Dir},D}) \iff \mbz\in \overline{D}.
	\]
\end{theorem}
\begin{proof}
	Firstly, let $\mbz\in \overline{D}$ and $v(\mbx)\coloneqq-2 \kappa^2G(\mbx,\mbz)$. Define
	$f_1= v|_{\p D}$ and $f_2= \p _{\mbn}v|_{\p D}$.
	Then $(f_1,f_2)^\top \in H^{3/2}(\p D)\times H^{1/2}(\p D)$, which follows directly from the fact that $G(\cdot ,\mbz) \in H^2(D)$ (see \cite[Lemma 2.2]{bourgeois2020linear}). Note that the far-field pattern $v^{\infty}$ of $v$ is given by $-2 \kappa^2 G^{\infty}(\cdot,\mbz)$, i.e., $\phi_{\mbz}$. So $\mcB_{{\rm Dir},D}(f_1,f_2)^\top=\phi_{\mbz}$.
	
	Secondly, let $\mbz\not \in \overline{D}$. Assume that there exists $(f_1,f_2)^\top \in  H^{3/2}(\p D)\times H^{1/2}(\p D)$ such that $\phi_{\mbz}=\mcB_{{\rm Dir},D}(f_1,f_2)^\top$. Denote by  $v$ the solution of \eqref{eq:bvp-1}-\eqref{eq:bvp-3} with Dirichlet data $(f_1,f_2)^\top$ and by $v^{\infty}$ its far-field pattern. Then $v^{\infty}=\phi_{\mbz}$,
	which together with Rellich's Lemma  \cite[Lemma 4.1]{dong2024uniqueness} and the unique continuation principle implies $v_{\rm pr}(\mbx)=\Phi_{\kappa}(\mbx,\mbz)$ holds in $\bbR^2\setminus(\overline{D}\cup \{\mbz\})$, where $v_{\rm pr}=-1/(2 \kappa^2)(\Delta v - \kappa^2 v)$ is the propagation part of $v$
	(see \cite{bourgeois2020well}). It is obvious that $v$ is analytic in $D^{\rm c}$ and so is $v_{\rm pr}$, which is a contradiction because $\Phi_{\kappa}(\mbx,\mbz)$ is singular at $\mbx=\mbz$. The proof is complete.
\end{proof}

\begin{remark}
	We want to emphasize that the result of the above theorem is slightly different from those of acoustic and elastic scattering problems, which mainly arises from the regularity of the fundamental solution. When $\mbz\in \p D$,  $G(\cdot ,\mbz)$ belongs to $H^2(D)$, whereas  $\Phi_{\kappa}(\cdot ,\mbz)$ fails to  belong to $H^1(D)$.
\end{remark}

Theorems \ref{thm:range-identities}, \ref{thm:solution-range} and Picard's theorem (see Theorem 4.8 in \cite{colton2019inverse}) yield one of the central results of this paper immediately.
\begin{theorem}\label{thm:indicator-d}
	Assume that $\kappa$ is not a clamped transmission eigenvalue. Then,
	\begin{align*}
		\mbz \in \overline{D}& \iff  \phi_{\mbz} \in {\rm Ran}((F_D^{*}F_D)^{1/4} ) \iff
		W_1(\mbz)\coloneqq \left[
		\sum\limits_j \dfrac{|( \phi_{\mbz} ,\psi_{1,j}) |^2 }{| \lambda_{1,j} | }      \right]^{-1} >0 \\
		& \iff  \phi_{\mbz} \in {\rm Ran}(F_{D,\#}^{1/2} ) \iff
		W_2(\mbz)\coloneqq \left[
		\sum\limits_j \dfrac{|( \phi_{\mbz} ,\psi_{2,j}) |^2 }{| \lambda_{2,j} | }      \right]^{-1} >0.
	\end{align*}
	Here, $\phi_{\mbz}$ is given in Theorem \ref{thm:solution-range},  $(\lambda_{1,j},\psi_{1,j},\tilde{\psi}_{1,j})$ is a singular system of $F_D$ and $(\lambda_{2,j},\psi_{2,j})$ is an eigensystem of $F_{D,\#}$.
\end{theorem}

As explained  in \cite{arens2009linear} and \cite[Theorem 5.42]{colton2019inverse},
Theorem \ref{thm:indicator-d} can immediately provide a rigorous theoretical foundation for the far-field linear sampling method for the inverse biharmonic scattering problem.

\subsection{The case of Neumann obstacle}
This subsection is concerned with the biharmonic scattering in the Neumann case. In this case, we define the \emph{data-to-pattern operator} $\mcB_{{\rm Neu},D} \colon  H^{-3/2}(\p D)\times H^{-1/2}(\p D)\to L^{2}(\bbS )$ such that for $(f_1,f_2)^\top\in H^{-3/2}(\p D)\times H^{-1/2}(\p D)$, $\mcB_{{\rm Neu},D}(f_1,f_2)^\top=v^{\infty}$, where $v^{\infty}$ is the far-field pattern of $v$ and $v$ is the solution to the exterior Neumann boundary value problem 
\begin{align}
	\label{eq:gbvp-n-1}
	\Delta^2 v- \kappa^4v &= 0 \quad \mathrm{ in}\; D^{\rm c},\\
	\label{eq:gbvp-n-2}
	( Mv, N v)&= (f_2,f_1)  \quad \mathrm{on}\; \p D,\\
	\label{eq:gbvp-n-3}
	\lim\limits_{r\to\infty}\sqrt{r}\left(\frac{\p v}{\p r}
	-\mathrm{i}\kappa v \right)&=0,\quad r=| \mbx |.
\end{align}
By Theorem 2.1 of \cite{bourgeois2020well}, the operator $\mcB_{{\rm Neu},D}$ is well-defined  except when $\kappa\in \mathcal{K}_0$. Under the assumption $\kappa\not \in \mathcal{K}_0$, we can get the properties of $\mcB_{{\rm Neu},D}$.
\begin{theorem} \label{thm:solution-operator-n}
	The operator $\mcB_{{\rm Neu},D}\colon  H^{-3/2}(\p D)\times H^{-1/2}(\p D)\to L^{2}(\bbS )$ is compact with dense range.
\end{theorem}
\begin{proof}
	The proof is similar to that of Theorem  \ref{thm:solution-operator}. To avoid repetition, we give only a sketch of the proof.
	Choosing a sufficiently large $R>0$ such that $\overline{D}\subset B(\bm{ 0 },R)$, we obtain the factorization $\mcB_{{\rm Neu},D}=\mcB_{D,2}\mcB_{{\rm Neu},D,1}$,
	where the operator $\mcB_{{\rm Neu},D,1}\colon H^{-3/2}(\p D)\times H^{-1/2}(\p D)\to (C(\p B(\bm{ 0 },R) ))^4$ is  given by
	\begin{equation*}
		\mcB_{{\rm Neu},D,1}(f_1,f_2)^\top\coloneqq (v|_{\p B(\bm{ 0 },R) },\p _{\mbn} v|_{\p B(\bm{ 0 },R) }, Mv|_{\p B(\bm{ 0 },R) }, Nv|_{\p B(\bm{ 0 },R) })^\top,
	\end{equation*}
	where $v$ is the solution to \eqref{eq:gbvp-n-1}-\eqref{eq:gbvp-n-3}, and  $\mcB_{D,2}\colon (C(\p B(\bm{ 0 },R) ))^4 \to L^2(\bbS )$ is defined in the proof of  Theorem  \ref{thm:solution-operator}.
	The interior regularity for the elliptic equation implies that the operator $\mcB_{{\rm Neu},D,1}$ is compact. Combining  this and the above factorization of $\mcB_{{\rm Neu},D}$, we conclude that the operator $\mcB_{{\rm Neu},D}$ is compact.
	
	Next, it suffices to prove the injectivity of $\mcB_{{\rm Neu},D}^{*}$.
	Set $U= U^{\rm sc}+U^{\rm in}$, where $U^{\rm sc}$ is defined as the solution to \eqref{eq:gbvp-n-1}-\eqref{eq:gbvp-n-3} corresponding to $(f_1,f_2)=-(N_{\mbx}U^{\rm in},M_{\mbx}U^{\rm in})$ and $U^{\rm in}$ have been defined in the proof of Theorem \ref{thm:solution-operator}. Applying the argument of Theorem  \ref{thm:solution-operator}, we obtain that $\mcB_{{\rm Neu},D}^{*} g = -1/(2\kappa^2) (\overline{U}, \p _{\mbn}\overline{U})^\top$.
	Let $\mcB_{{\rm Neu},D}
	^{*}g={\bm 0}$ for some $g \in L^{2}(\bbS )$. It follows from the boundary conditions $U$ on $\p D$ that $U|_{\p D}=\p_{\mbn}U|_{\p D}=MU|_{\p D}=NU|_{\p D}=0$. Further,  $U$ vanishes in $\bbR^2\setminus\overline{D}$ by Holmgren's uniqueness theorem, which implies that $U^{\rm in}$ satisfies the Sommerfeld radiation condition. This in combination with the fact that $U^{\rm in}$ satisfies the Helmholtz equation in $\bbR^2$, yields that  $g\equiv 0$, which finishes the proof.
\end{proof}

Analogously to the operator $\mcS_{{\rm Dir},D}$, define the operator $\mcS_{{\rm Neu},D}\colon H^{3/2}(\p D)\times H^{1/2}(\p D)\to H^{-3/2}(\p D)\times H^{-1/2}(\p D)$ by 
\begin{equation}
	\label{eq:op-s-n}
	\displaystyle
	\mcS_{{\rm Neu},D}
	\begin{pmatrix} 
		\tau \\ \sigma 
	\end{pmatrix} ({\mbx})
	\coloneqq \begin{pmatrix}    \displaystyle
		N_{\mbx} \int_{ \p D  } \big[ N_{\bm y} G(\mbx,\bm{ y }) \tau(\bm{ y }) +  M_{\bm y}G(\mbx,\bm{ y })\sigma(\bm{ y })\big]  \,\,{\rm d}  s(\bm{ y }) \\[5mm]
		\displaystyle
		M_{\mbx} \int_{ \p D  } \big[ N_{\bm y} G(\mbx,\bm{ y }) \tau(\bm{ y }) +  M_{\bm y}G(\mbx,\bm{ y })\sigma(\bm{ y })\big] \,\,{\rm d} s(\bm{ y })
	\end{pmatrix},
\end{equation}
for $\mbx\in  \p D$. Then, the factorization of far-field operator is given in the following theorem for the Neumann case.
\begin{theorem}
	The factorization
	\begin{equation}\label{eq:far-field-factorization-n}
		F_D=2\kappa^2 \mcB_{{\rm Neu},D}\mcS_{{\rm Neu},D}^*\mcB_{{\rm Neu},D}^*
	\end{equation}
	holds. 
\end{theorem}
\begin{proof}
	Similarly, we introduce the \emph{data-to-obstacle operator}  $\mathcal{H}_{{\rm Neu},D}\colon L^2(\bbS )\to H^{-3/2}(\p D)\times H^{-1/2}(\p D)$ by
	\begin{equation}
		\label{eq:op-incident-n}
		\mathcal{H}_{{\rm Neu},D} g(\mbx ) \coloneqq
		\begin{pmatrix}     \displaystyle
			N_{\mbx}\int_{ \bbS   } e^{{\rm i}\kappa \mbx\cdot \mbd} g(\mbd) \,\,{\rm d} s(\mbd)\\[5mm]
			\displaystyle
			M_{\mbx}\int_{ \bbS   } e^{{\rm i}\kappa \mbx\cdot \mbd} g(\mbd) \,\,{\rm d} s(\mbd)
		\end{pmatrix}
		\quad \text{for}\; \mbx\in \p D \;\text{and} \; g\in L^2(\bbS ).
	\end{equation}
	By the superposition principle, we deduce that $F_D=-\mcB_{{\rm Neu},D}\mathcal{H}_{{\rm Neu},D}$.
	Noting that the adjoint  $\mathcal{H}_{{\rm Neu},D}^*\colon H^{3/2}(\p D)\times H^{1/2}(\p D)\to L^2(\bbS )$ is given by
	\begin{equation*}
		\mathcal{H}_{{\rm Neu},D}^{*} \begin{pmatrix} \tau\\
			\sigma \end{pmatrix} (\hat{\mbx}) 
		= \int_{ \p D  } \left( 
		N_{\bm y} e^{-{\rm i}\kappa \hat{\mbx}\cdot \bm{ y }}  \tau(\bm{ y })
		+ M_{\bm y} e^{-{\rm i}\kappa \hat{\mbx}\cdot \bm{ y }}  \sigma(\bm{ y })
		\right) \,\,{\rm d}  s(\bm{ y })  
		\quad \hat{\mbx} \in \bbS,
	\end{equation*}
	we see that 
	\begin{equation*}
		\mathcal{H}_{{\rm Neu},D}^{*} 
		(\tau,
		\sigma )^\top= v^{\infty},
	\end{equation*}
	where $v$ is defined by 
	\begin{equation*}
		v(\mbx) = -2\kappa^2 \int_{ \p D  } \left( N_{\bm y} G(\mbx, \bm{ y })  \tau(\bm{ y }) 
		+ M_{\bm y} G(\mbx, \bm{ y })  \sigma(\bm{ y })\right) \,\,{\rm d}  s(\bm{ y })  
		\quad {\mbx} \in D^{\rm c}.
	\end{equation*}
	We observe that $(N v,M v)^\top=-2 \kappa^2 \mcS_{{\rm Neu},D} (\tau,\sigma)^\top$.
	Then from the definition of $\mcB_{{\rm Neu},D}$ we get $\mathcal{H}_{{\rm Neu},D}^{*}=-2\kappa^2 \mcB_{{\rm Neu},D}\mcS_{{\rm Neu},D}$. Thus $\mathcal{H}_{{\rm Neu},D}=-2\kappa^2 \mcS_{{\rm Neu},D}^*\mcB_{{\rm Neu},D}^*$ and  $F_D=2\kappa^2 \mcB_{{\rm Neu},D}\mcS_{{\rm Neu},D}^*\mcB_{{\rm Neu},D}^*$. Therefore, the proof is completed.
\end{proof}

We introduce   $\tilde{\mcS}_{{\rm Neu},D} \colon H^{3/2}(\p D)\times H^{1/2}(\p D)\to H^{-3/2}(\p D)\times H^{-1/2}(\p D)$ by
\begin{equation*}
	\displaystyle
	\tilde{\mcS}_{{\rm Neu},D}\begin{pmatrix} \tau \\ \sigma \end{pmatrix} ({\mbx})
	\coloneqq \begin{pmatrix}     \displaystyle
		N_{\mbx}\int_{ \p D  } \left( N_{\bm y}\tilde{G}(\mbx,\bm{ y }) \tau(\bm{ y }) + 
		M_{\bm y}\tilde{G}(\mbx,\bm{ y })\sigma(\bm{ y })\right) \,\,{\rm d}  s(\bm{ y }) \\[5mm]
		\displaystyle
		M_{\mbx}
		\int_{ \p D  } \left( N_{\bm y} \tilde{G}(\mbx,\bm{ y }) \tau(\bm{ y }) 
		+  M_{\bm y}\tilde{G}(\mbx,\bm{ y }) \sigma(\bm{ y })\right) \,\,{\rm d} s(\bm{ y })
	\end{pmatrix}
	\, \mbx\in  \p D.
\end{equation*}
We collect some properties of $\mcS_{{\rm Neu},D}$ and $\tilde{\mcS}_{{\rm Neu},D}$.
\begin{theorem}\label{thm:middle-n}
	Let $\mathcal{Y}:=H^{3/2}(\p D)\times H^{1/2}(\p D)$. Then we have
	\begin{enumerate}
		\item[(a)] The operator $\mcS_{{\rm Neu},D}$ is an isomorphism from $\mathcal{Y}\to\mathcal{Y}^*$ if $\kappa$ is not a fourth root of a Neumann eigenvalue of the bi-Laplacian operator $\Delta^2$ in $D$.
		\item[(b)]  $\mathrm{Im}\, \left\langle
		\mcS_{{\rm Neu},D}\begin{pmatrix} \tau\\ \sigma\end{pmatrix},
		\begin{pmatrix} \tau\\ \sigma \end{pmatrix}
		\right\rangle <0$  for all   $\begin{pmatrix} \tau\\ \sigma\end{pmatrix} \in \overline{{\rm Ran}(\mcB_{{\rm Neu},D}^{*})}$ with $\begin{pmatrix} \tau\\ \sigma \end{pmatrix}\not= \bm{ 0 }$ if $\kappa$ is not a free transmission eigenvalue.
		
		\item[(c)] $\tilde{\mcS}_{{\rm Neu},D}\colon \mathcal{Y}  \to \mathcal{Y}^*$ is self-adjoint and  positive coercive.
		
		\item[(d)]  $\mcS_{{\rm Neu},D} - \tilde{\mcS}_{{\rm Neu},D}\colon  \mathcal{Y}  \to \mathcal{Y}^*$ is compact.
	\end{enumerate}
\end{theorem}
\begin{proof}
	(a) follows from Section 2.4 in \cite{bourgeois2020linear}.
	
	(b) We start by  proving
	\begin{equation*}
		\mathrm{Im}\, \left\langle
		\mcS_{{\rm Neu},D}\begin{pmatrix}
			\tau\\ \sigma
		\end{pmatrix}
		, \begin{pmatrix}
			\tau\\ \sigma
		\end{pmatrix}\right\rangle \leq 0\quad
		\forall  \begin{pmatrix}\tau\\ \sigma\end{pmatrix} \in \mathcal{Y}.
	\end{equation*}
	To this end, define the following potential 
	\begin{equation}\label{eq:potential-n}
		v(\mbx)= \int_{ \p D  } \left[ N_{\bm y }G(\mbx, \bm{ y }) \tau(\bm{ y }) +M_{\bm y} G(\mbx,\bm{ y })\sigma(\bm{ y })
		\right]  \,\,{\rm d}s(\bm{ y }) \quad \mbx\not \in \p D.
	\end{equation}
	Then $v \in H^{2}_{\mathrm{loc} }(\bbR^2\setminus \p D)$ is the solution to the biharmonic wave equation in $\bbR^2\setminus \p D$.  From the jump relations of the single- and double-layer potentials, we conclude
	\begin{align}
		\label{eq:jump-n-1}
		\tau& =v|_{-} - v|_{+}\quad
		\sigma = \p_{\mbn} v|_{-}  - \p_{\mbn}  v|_{+}   \quad \mathrm{on} ~ \p D, \\
		\label{eq:jump-n-2}
		0 & = N\, v|_{+} - N\, v|_{-}\quad
		0 = M\, v|_{+} - M\, v|_{-}  \quad \mathrm{on} ~ \p D.
	\end{align}
	Green's formulas in $D$ and $D_R\coloneqq\{ \mbx \in D^{\rm c}: | \mbx | < R\}$ imply
	\begin{align*}
		&\left\langle
		\mcS_{{\rm Neu},D}\begin{pmatrix}
			\tau\\ \sigma
		\end{pmatrix}
		, \begin{pmatrix}
			\tau\\ \sigma
		\end{pmatrix}\right\rangle
		=   \left\langle
		\begin{pmatrix}
			N v\\
			Mv
		\end{pmatrix},
		\begin{pmatrix}
			v|_- - v|_+\\ \p _{\mbn} v|_- - \p _{\mbn} v|_+
		\end{pmatrix}   \right\rangle\\
		={}&  \int_{ D\cup D_R  } \left( a(v,\overline{v}) -\kappa^4 | v |^2\right) \,\,{\rm d} \mbx
		-  \int_{ | \mbx | = R  } [\overline{v} \,Nv+ \p _{\mbn}\overline{v}\, Mv   ] \,\,{\rm d} s(\mbx).
	\end{align*}
	Using Lemma 2.3 in \cite{zhu2025direct},  the imaginary part has the following expression
	\begin{align}\notag
		\mathrm{Im}\, \left\langle
		\mcS_{{\rm Neu},D}\begin{pmatrix}
			\tau\\ \sigma
		\end{pmatrix}
		, \begin{pmatrix}
			\tau\\ \sigma
		\end{pmatrix}\right\rangle
		&=  -\mathrm{Im}\,  \int_{ | \mbx | = R  }
		[\overline{v}\, Nv+ \p _{\mbn}\overline{v}\, Mv   ]
		\,\,{\rm d} s(\mbx)\\
		\label{eq:double-imaginary}
		&=  -\frac{\kappa^{2}}{4\pi}\int_{\bbS}|v^{\infty}|^2\mathrm{d}s \leq 0.
	\end{align}
	
	Next, it suffices to prove that \[
	\mathrm{Im}\, \left\langle
	\mcS_{{\rm Neu},D}
	\begin{pmatrix}
		\tau\\ \sigma
	\end{pmatrix}
	, \begin{pmatrix}
		\tau\\ \sigma
	\end{pmatrix}\right\rangle = 0 \quad \text{for some} \quad \begin{pmatrix}
		\tau\\ \sigma
	\end{pmatrix} \in   \overline{{\rm Ran}(\mcB_{{\rm Neu},D}^{*})}\]
	implies  that $( \tau, \sigma)^\top= \bm{ 0 }$. 
	Using \eqref{eq:double-imaginary}, we have $v^{\infty}=0$ and thus $v_{\rm pr}=0$ and $v=v_{\rm ev}$ hold in $D^{\rm c}$ by Rellich's lemma. It is clear that  $ (N_{\cdot}\Phi_{{\rm i}\kappa}(\cdot,\mbz), M_{\cdot}\Phi_{{\rm i}\kappa}(\cdot,\mbz) )^\top\in {\rm Ker}(\mcB_{{\rm Neu},D})$ holds for $\mbz\in D$. Then the space relation  $\overline{{\rm Ran}(\mcB_{{\rm Neu},D}^{*})}=\prescript{a}{}{[{\rm Ker}(\mcB_{{\rm Neu},D})]}$ implies that
	\begin{equation*}
		\left \langle  
		\begin{pmatrix}
			N_{\cdot}\Phi_{{\rm i}\kappa}(\cdot,\mbz)\\
			M_{\cdot}\Phi_{{\rm i}\kappa}(\cdot,\mbz)
		\end{pmatrix},  
		\begin{pmatrix}
			\tau\\  
			\sigma
		\end{pmatrix}\right\rangle=0 \quad \text{for all}\;\mbz\in D,
	\end{equation*}
	i.e.,
	\begin{equation*}
		\int_{ \p D  }\big[  N_{\bm y}\Phi_{{\rm i}\kappa}(\bm{ y } ,\mbz)\overline{\tau}(\bm{ y })
		+   M _{\bm{ y }}\Phi_{{\rm i}\kappa}(\bm{ y } ,\mbz)\overline{\sigma}(\bm{ y }) \big]
		\,\,{\rm d}s(\bm{ y }) =0 \quad \text{for all} \quad \mbz \in  D.
	\end{equation*}
	This indicates that $\overline{v}_{\rm ev}=0$ and $\overline{v}=\overline{v}_{\rm pr}$ hold in $D$. Let $w_1=\overline{v}|_{D^{\rm c}}=\overline{v}_{\rm ev}|_{D^{\rm c}}$  and  $w_2=-\overline{v}|_D=-\overline{v}_{\rm pr}|_D$. Then the jump relations \eqref{eq:jump-n-2} yields that $(w_1,w_2)$ satisfies the free transmission eigenvalue problem \eqref{eq:free}
	\begin{equation*}
		\begin{aligned}
			\Delta w_1 - \kappa^2 w_1 & = 0   \quad \mathrm{ in } ~ D^{\rm c}, \\
			\Delta w_2 + \kappa^2 w_2 & = 0   \quad \mathrm{ in } ~ D, \\
			(M(w_1 + w_2),	N(w_1 +  w_2)) & = {\bm 0} \quad \mathrm{ on } ~ \p D.
		\end{aligned}
	\end{equation*}
	Since  $\kappa$ is not a free transmission eigenvalue, the functions $(w_1,w_2)^\top={\bm 0}$ and $\overline{v}= 0$ in $\bbR^2$, which, combined with the jump relations \eqref{eq:jump-n-1}, shows $(\tau,\sigma)^\top={\bm 0}$.
	
	(c) Let \[
	\mathbb{G}_{\rm Neu}(\mbx,\bm{ y })\coloneqq 
	\begin{pmatrix}
		N_{\mbx}N_{\bm y}\tilde{G}(\mbx,\bm{ y })
		&   \displaystyle  N_{\mbx}M_{\bm y} \tilde{G}(\mbx,\bm{ y })\\[3mm]
		\displaystyle
		M_{\mbx}N_{\bm y} \tilde{G}(\mbx,\bm{ y })
		&   M_{\mbx}M_{\bm y} \tilde{G}(\mbx,\bm{ y })
	\end{pmatrix}.
	\]
	Since $\mathbb{G}_{\rm Neu}(\mbx,\bm{ y }) = \mathbb{G}_{\rm Neu}(\bm{ y },\mbx)^\top$ holds and the function $\tilde{ G}(\mbx,\bm{ y })$ is real-valued, the operator $\tilde{\mcS}_{{\rm Neu},D}$ is self-adjoint.

	We briefly  give a proof for the isomorphism of operator $\tilde{\mcS}_{\mathrm{ Neu },D }$. Assume that there exists $(\tau,\sigma)^\top \in \mathcal{Y}$ such that $\tilde{\mcS}_{\mathrm{ Neu },D }(\tau,\sigma)^\top = {\bm 0}$. This shows that $(Nv,Mv)^\top=\bm{ 0 }$ on $\p D$, where the function $v$ is defined by \eqref{eq:potential-n}. The variational method can easily imply that for any $(f_1,f_2)^{\top}\in \mathcal{Y}^*$, the interior Neumann problem
	\begin{align}
		\label{eq:INP-1}
		 (\Delta^2 + 1) u &= 0 \quad \mathrm{ in }~ D,\\
		\label{eq:INP-2}
		(Mu,Nu) &= (f_2,f_1)\quad \mathrm{ on }~  \p D.
	\end{align}
	exists a unique solution $u \in H^{2}(D)$ and the exterior Neumann problem
	\begin{align}
		\label{eq:ENP-1}
		(\Delta^2 +  1) u &= 0 \quad \mathrm{ in }~ D^c,\\
		\label{eq:ENP-2}
		(Mu,Nu) &= (f_2,f_1)\quad \mathrm{ on }~  \p D.     
	\end{align}
	exists a unique solution $u \in H^{2}(D^c)$. Note that the function $v|_D$ (resp. $v|_{D^c}$) satisfies the interior (resp. exterior)  Neumann problem \eqref{eq:INP-1}-\eqref{eq:INP-2} (resp. \eqref{eq:ENP-1}-\eqref{eq:ENP-2}) with homogeneous boundary data. From the well-posedness of the above interior and exterior Neumann problems, we have $v|_{D}=v|_{D^c}=0$, which together with jump relations \eqref{eq:jump-n-1} indicates $(\tau,\sigma)^\top = {\bm 0}$. Given $(\phi,\psi)^{\top}\in  \mathcal{Y}^*$, we will prove that there exists  $(\tau,\sigma)^\top \in \mathcal{Y}$ such that $\tilde{\mcS}_{\mathrm{ Neu },D }(\tau,\sigma)^\top = (\phi,\psi)^{\top}$. Then there exist $u^{-}\in H^{2}(D)$ satisfying \eqref{eq:INP-1}-\eqref{eq:INP-2} with $(f_1,f_2)=(\phi,\psi)$  and $u^{+}\in H^{2}(D^{c})$ satisfying \eqref{eq:ENP-1}-\eqref{eq:ENP-2} with $(f_1,f_2)=(\phi,\psi)$. Define the function $u$ such that $u|_{D}=u^{-}$ and $u|_{D^c}=u^{+}$. Applying Green formulas in $D$ and $D^c$ can yield the integral representation of the function $u$:
	\begin{equation*}
		u(\mbx)= \int_{ \p D  } \left[ N_{\bm y }\tilde{G}(\mbx, \bm{ y }) [u](\bm{ y }) +M_{\bm y} \tilde{G}(\mbx,\bm{ y }) [\p_{\mbn} u](\bm{ y })
		\right]  \,\,{\rm d}s(\bm{ y }) \quad \mbx\not \in \p D,
	\end{equation*}
	where 
	$[u] =u|_{-} - u|_{+}$
	and  $[\p_{\mbn} u]= \p_{\mbn} u|_{-}  - \p_{\mbn}  u|_{+}$.
	We observe that $(Nu,Mu)$  coincide with $(\phi,\psi)$. Setting $\tau=[u]$ and $\sigma=[\p _{\mbn}u]$ gives  $\tilde{\mcS}_{\mathrm{ Neu },D }(\tau,\sigma)^\top = (\phi,\psi)^{\top}$.
	
	Next, we turn to proving the coercivity of the operator $\tilde{\mcS}_{\mathrm{ Neu },D }$.
	Given $(\tau,\sigma)^\top\in\mathcal{Y}$, define the function $v$ according to \eqref{eq:potential-n} with $G$ replaced by $\tilde{G}$.
	Using the  fact that $v$ decays exponentially as $| \mbx | \to \infty$,  the Green's formulas in  $D$ and $D_R$ give 
	\begin{align*}
		\left\langle
		\tilde{\mcS}_{{\rm Neu},D}\begin{pmatrix}
			\tau\\ \sigma
		\end{pmatrix}
		, \begin{pmatrix}
			\tau\\ \sigma
		\end{pmatrix}\right\rangle
		&=  \int_{ D\cup D^{\rm c}}   \left( a(v,\overline{v})  +  | v |^2\right) \,\,{\rm d} \mbx
			\geq c \big(\| v \|_{H^2(D)}^2 + \| v \|_{H^2(D^c)}^2\big),
	\end{align*}
	where we let $R$ tend to infinity and $c>0$ is a constant.  This, combined with the trace theorem (see Lemma 5.7.1 in \cite{Hsiao2021boundary}) and the boundedness of $\tilde{\mcS}_{{\rm Neu},D}^{-1}$ implies that there exists a constant $c>0$ such that 
	\begin{align*}
		\left\langle
		\tilde{\mcS}_{{\rm Neu},D}\begin{pmatrix}
			\tau\\ \sigma
		\end{pmatrix}
		, \begin{pmatrix}
			\tau\\ \sigma
		\end{pmatrix}\right\rangle
		& \geq  c \big(\| v \|_{H^2(D)}^2 + \| v \|_{H^2(D^c)}^2\big)
		\geq c \Big\|\begin{pmatrix}
			Nv\\  M v
		\end{pmatrix} \Big\|_{\mathcal{Y}^*}^2\\
		& \geq c \Big\|  \tilde{\mcS}_{{\rm Neu},D}\begin{pmatrix}
			\tau\\ \sigma
		\end{pmatrix}\Big\|_{\mathcal{Y}^*}^2 
		\geq c \Big\|\begin{pmatrix}
			\tau\\ \sigma
		\end{pmatrix}\Big\|_{\mathcal{Y}}^2.
	\end{align*}

	(d) follows from the proof of (d) in Theorem \ref{thm:middle}.
\end{proof}

Theorems \ref{thm:far-field} and \ref{thm:middle-n} ensure that the factorization \eqref{eq:far-field-factorization-n} satisfies the conditions of  Theorem \ref{thm:range-1} by identifying $H=L^2(\bbS)$, $X=H^{-3/2}(\p D)\times H^{-1/2}(\p D)$, $F=F_D$, $B=\mcB_{{\rm Neu},D}$ and $A=2 \kappa^2\mcS_{{\rm Neu},D}^{*}$. Similarly, with $Y=L^2(\bbS)$, $X=H^{-3/2}(\p D)\times H^{-1/2}(\p D)$, $F=-F_D$, $G=\mcB_{{\rm Neu},D}$, $T=-2 \kappa^2\mcS_{{\rm Neu},D}^{*}$, $T_0=-2 \kappa^2\tilde{\mcS}_{{\rm Neu},D}^{*}$ and $t=\pi$, the hypotheses of Theorem \ref{thm:range-2} are fulfilled by virtue of Theorems \ref{thm:far-field} and \ref{thm:middle-n}.  Accordingly, the following range identities are obtained.
\begin{theorem}\label{thm:range-identities-n}
	Assume that $\kappa\notin\mathcal{K}_0$ and that $\kappa$ is not a free transmission eigenvalue. Then  
	\begin{equation*}
		{\rm Ran}((F_D^{*}F_D)^{1/4} )={\rm Ran}(F_{D,\#}^{1/2})={\rm Ran}(\mcB_{{\rm Neu},D}).
	\end{equation*}
\end{theorem}

\begin{theorem}\label{thm:solution-range-n}
	Assume that $\kappa\notin\mathcal{K}_0$. Then, for any $\mbz\in \bbR^2$, 
	we have
	\[
	\phi_{\mbz} \in {\rm Ran}(\mcB_{{\rm Neu},D}) \iff \mbz\in \overline{D}.
	\]
\end{theorem}
\begin{proof}
	The proof is analogous to that of Theorem \ref{thm:solution-range} and is omitted here for brevity.
\end{proof}

The combination of Theorems \ref{thm:range-identities-n},  \ref{thm:solution-range-n},  and Picard's Theorem (see Theorem 4.8 in \cite{colton2019inverse}) gives the following theorem for the Neumann case.
\begin{theorem}\label{thm:indicator-n}
	Assume that $\kappa\notin\mathcal{K}_0$ and that $\kappa$ is not a free transmission eigenvalue. Define the singular system $(\lambda_{1,j},\psi_{1,j},\tilde{\psi}_{1,j})$, the eigensystem $(\lambda_{2,j},\psi_{2,j})$, and $W_1,W_2$ as in Theorem~\ref{thm:indicator-d}, using the far-field operator for the Neumann case.  Then,
	\begin{align*}
			\mbz \in \overline{D} & \iff  \phi_{\mbz} \in {\rm Ran}((F_D^{*}F_D)^{1/4} ) \iff
			W_1(\mbz) >0 \\
			& \iff  \phi_{\mbz} \in {\rm Ran}(F_{D,\#}^{1/2} ) \iff
			W_2(\mbz)>0. 
	\end{align*}
	Here, $\phi_{\mbz}$ is given in Theorem \ref{thm:solution-range}. 
\end{theorem}

\section{The monotonicity method}
In this section, we use the factorization of far-field operator in the previous section to establish the monotonicity relation between the far-field operator $F_D$ and a probing operator $H^{*}_{\p B}H_{\p B}$, which will be used to locate the support of the Dirichlet or Neumann obstacle.

Let us start with an ordering relation (see \cite{harrach2019monotonicity}) and the general functional analysis theorem (see \cite[Theorem 4.2]{furuya2021remarks}) for the monotonicity method.

\begin{definition}
	Let $A,B \colon H\to H$ be self-adjoint compact operators on a
	Hilbert space $H$.
	We write $A\leq_{\rm fin} B$ if $B-A$ has only finitely many negative eigenvalues. 
\end{definition}

\begin{theorem}\label{thm:MM}
	Let $X\subset U \subset X^{*}$ and
	$\tilde{X}\subset\tilde{U}\subset \tilde{X}^{*}$ be Gelfand triples
	with dense embeddings, where $U$ and $\tilde{U}$ are
	Hilbert spaces,
	with the identification $U^*\equiv U$ and $\tilde{U}^*=\tilde{U}$,
	and $X$ and $\tilde{X}$ are reflexive Banach spaces.
	Furthermore, let $Y$ be a Hilbert space with the
	identification $Y^*\equiv Y,$
	and let $F\colon Y\to Y$, $\tilde{F}\colon Y\to Y$, $G\colon X\to Y$, $\tilde{G}\colon \tilde{X}\to Y$, $T\colon X^{*}\to X$, $\tilde{T}\colon \tilde{X}^{*}\to \tilde{X}$ be linear bounded operators such that
	\begin{equation*}
		F=GTG^{*},\quad \tilde{F} = \tilde{G} \tilde{T} \tilde{G}^{*}.
	\end{equation*}
	\begin{enumerate}
		\item[(1)] Assume that
		\begin{enumerate}
			\item [(1a)] $\mathrm{Re}\;T$ has the form $\mathrm{Re}\;T= T_0+K$ where $T_0\colon X^{*}\to X$ is a positive coercive operator and $K\colon X^{*}\to X$ is a self-adjoint compact operator.
			\item [(1b)] There exists a compact operator $R\colon \tilde{X}\to X$ such that $\tilde{G}=GR$. 
		\end{enumerate}
		Then,
		\begin{equation*}
			\mathrm{Re}\;\tilde{F} \leq_{\rm fin} \mathrm{Re}\;F.
		\end{equation*}
		\item[(2)] Assume that
		\begin{enumerate}
			\item[(2a)] $\mathrm{Re}\;\tilde{T}$ has the form $\mathrm{Re}\;\tilde{T}=\tilde{T_0} +\tilde{K}$ where $\tilde{T}_0\colon \tilde{X}^{*}\to\tilde{X}$ is a positive coercive operator and  $\tilde{K}\colon \tilde{X}^{*}\to\tilde{X}$ is a self-adjoint compact operator.
			\item[(2b)] There exists an infinite-dimensional subspace $W$ in ${\rm Ran}(\tilde{G})$ such that $ W\cap {\rm Ran}(G)=\{ 0\}$.
		\end{enumerate}
		Then
		\begin{equation*}
			\mathrm{Re}\; \tilde{F} \not\leq_{\rm fin} \mathrm{Re}\; F.
		\end{equation*}
	\end{enumerate}
\end{theorem}

Let $B$ be a bounded domain with smooth boundary. Then we define the probing operator $H_{\p B}\colon L^2(\bbS)\to L^2(\p B)$
\begin{equation}\label{eq:HpB}
	H_{\p B} g(\mbx) = \int_{ \bbS   } e^{\mathrm{ i } \kappa \mbx\cdot  \mbd} g(\mbd) \,\,{\rm d}s(\mbd) \quad \mbx \in \p B.
\end{equation}

\begin{theorem}\label{thm:H-adjoint}
	Suppose that $B$ is a bounded and smooth domain and $\G \subset \p B$ is relatively open.
	Let $H_{\G}$ be $H_{\p B}$ with $\p B$ replaced by $\G$ in \eqref{eq:HpB}.
	Then, we have
	\begin{enumerate}
		\item [(a)] The dimension of ${\rm Ran}(H_{\G}^{*})$ is infinite.
		\item[(b)] ${\rm Ran}(H_{\G}^{*}) \cap {\rm Ran}(\mcB_{{\rm Dir},D})=        \{0\}$
		and ${\rm Ran}(H_{\G}^{*}) \cap {\rm Ran}(\mcB_{{\rm Neu},D})=        \{0\}$ if $\overline{\G} \cap D= \emptyset$ and $\mathbb{R}^2\setminus \overline{\Gamma \cup D}$ is connected.
	\end{enumerate}
\end{theorem}
\begin{proof}
	(a)   For the reader's convenience, we provide a detailed proof by
	adapting the approach in \cite{albicker2020monotonicity}. By modifying the boundary of $B$ away from $\Gamma$, we observe that $H_{\Gamma}$  remains invariant, and thus it suffices to assume that $\kappa^2$ is not a Dirichlet eigenvalue of $-\Delta$ in $B$.
	Define the data-to-pattern operator $\mcB_B^{\mathrm{acou}}\colon H^1(\p B)\to L^{2}(\bbS )$ of the exterior Dirichlet problem for the Helmholtz equation:
	\begin{equation*}
		\mcB_B^{\mathrm{acou}} \lambda = v^{\infty} \quad \mathrm{ for } \quad \lambda \in  H^{1}(\p B),
	\end{equation*}
	where $v^{\infty}$ is the far-field pattern of $v \in H_{\mathrm{ loc } }^1(B^{\rm c})$ satisfying
	\begin{align*}
		\Delta v+ \kappa^2 v & = 0 \quad \mathrm{ in } ~ B^{\rm c}, \\
		v & =  \lambda \quad \mathrm{ on } ~ \p B,  \\
		\frac{\p v}{\p r} - \mathrm{ i } \kappa v  & =  o(r^{-1/2}) \quad  \mathrm{ as } \quad r\to \infty.
	\end{align*}
	Clearly, the operator $\mcB_B^{\mathrm{acou}}$ is a bounded operator by the regularity of the exterior acoustic scattering problem. From the uniqueness of the exterior Helmholtz boundary value problem, the operator $\mcB_B^{\mathrm{acou}}$ is injective. The single-layer boundary operator $S_B^{\mathrm{acou} }\colon  L^2(\p B)\to H^1(\p B)$ for the Helmholtz equation is given by
	\begin{equation*}
		S_B^{\mathrm{acou} }\psi (\mbx)=\int_{ \p B  } \Phi_{\kappa}(\mbx,\bm{ y })\psi(\bm{ y }) \,\,{\rm d} s(\bm{ y }) \quad \mbx\in \p B,
	\end{equation*}
	which is injective according to the assumption that $\kappa^2$ is not a Dirichlet eigenvalue of $-\Delta$ in $B$.
	Besides, we  introduce the extension operator $\mathcal{E}\colon L^2(\G)\to L^2(\p B)$:
	\begin{equation*}
		\mathcal{E} f =
		\begin{cases}
			f \quad \mathrm{ on }  \quad \G,\\
			0 \quad \mathrm{ on } \quad \p B\setminus \G.
		\end{cases}
	\end{equation*}
	By the operator $\mathcal{E}$ and the relation $H_{\p B}^{*}=\mcB_B^{\mathrm{acou}}S_B^{\mathrm{acou}}$ (see \cite{colton2019inverse}), we find that
	\begin{equation*}
		H_{\G}^{*}= H_{\p B}^{*}\mathcal{E} = \mcB_B^{\mathrm{acou}}S_B^{\mathrm{acou}} \mathcal{E}.
	\end{equation*}
	It is obvious that the dimension of the range of $\mathcal{E}$ is infinite. This combined with the injectivity of $\mcB_B^{\mathrm{acou}}$ and $S_B^{\mathrm{acou}}$,  implies that the range of $H_{\G}^{*}$ is infinite-dimensional.
	
	(b) Let $h \in  {\rm Ran}(H_{\G}^{*}) \cap {\rm Ran}(\mcB_{{\rm Dir},D})$. There exist $f \in L^2(\G)$ and $(f_1,f_2)^\top \in H^{3/2}(\p D)\times H^{1/2}(\p D)$ such that
	\begin{equation*}
		h= H^{*}_{\G}f= \mcB_{{\rm Dir},D}
		\begin{pmatrix}
			f_1 \\ f_2
		\end{pmatrix}.
	\end{equation*}
	Hence $h = v_{\G}^{\infty}= w^{\infty}$, where $v_{\G} \in H^{1}_{\mathrm{ loc } }(\bbR^2\setminus\overline{\G})$ and $w \in H^{2}_{\mathrm{ loc } }(D^{\rm c})$ are solutions of
	\begin{equation*}
		\Delta v_{\G} +\kappa^2 v_{\G} =0 \quad \mathrm{ in } \quad \bbR^2\setminus \overline{\G} \quad
		\mathrm{ and } \quad \Delta^2 w -\kappa^4 w =0 \quad \mathrm{ in } \quad D^{\rm c},
	\end{equation*}
	respectively, and both satisfy the Sommerfeld radiation condition. By Rellich's lemma  \cite[Lemma 4.1]{dong2024uniqueness}, we obtain
	\begin{equation*}
		v_{\G}=w_{\rm pr} \quad \mathrm{ in } \quad \bbR^2\setminus \overline{D\cup \G},
	\end{equation*}
	where $w_{\rm pr}= -1/(2 \kappa^2)(\Delta w -\kappa^2 w)$. Note that $w_{\rm pr}$ satisfies the Helmholtz equation with wavenumber $\kappa$, i.e., $\Delta w_{\rm pr} + \kappa^2 w_{\rm pr}=0$ in $D^{\rm c}$. Define $\hat{w} \in H^{1}_{\mathrm{ loc } }(\bbR^2)$ by
	\begin{equation*}
		\hat{w} \coloneqq
		\begin{cases}
			v_{\G}=w_{\rm pr} \quad \mathrm{ in } \quad  \bbR^2\setminus \overline{D\cup \G},\\
			v_{\G} \quad \qquad \quad\mathrm{ in } \quad  D,\\
			w_{\rm pr}\quad \qquad \,\;\mathrm{ on }\quad  \G.
		\end{cases}
	\end{equation*}
	We observe that $\hat{w} $ satisfies the Helmholtz equation in the whole $\bbR^2$ and the Sommerfeld radiation condition. Thus $\hat{w} $ vanishes in $\bbR^2$, which yields that the far-field pattern $h$ is identically zero.
	
	For the Neumann obstacle case,  the proof of  ${\rm Ran}(H_{\G}^{*}) \cap {\rm Ran}(\mcB_{{\rm Neu},D})=   \{0\}$ is omitted as  it  proceeds along the same lines as that for the Dirichlet obstacle.
	Consequently, the proof is finished. 
\end{proof}
We collect the above results and establish the monotonicity method.
\begin{theorem}\label{thm:MM-d}
	Suppose that $B\subset \bbR^2$ is a bounded and smooth domain.
	If $D$ is a Dirichlet obstacle, then the following holds:
	\begin{equation*}
		B \subset D \iff H_{\p B}^{*} H_{\p B} \leq_{\mathrm{ fin} } - \mathrm{Re}\;F_D.
	\end{equation*}
\end{theorem}
\begin{proof}
	Assume  $B \subset D$. Define $R\colon L^2(\p B)\to H^{3/2}(\p D)\times H^{1/2}(\p D)$ by
	$R \phi =  ( v|_{\p D}, \p _{\mbn}v |_{\p D})^{\top}$, 
where $v$ is the single-layer potential of the biharmonic wave equation defined by
\[ v(\mbx)=- 2 \kappa^2SL_{\p B}\phi(\mbx)
\coloneqq - 2 \kappa^2\int_{ \p B  } G(\mbx,\bm{ y })\phi(\bm{ y }) \,\,{\rm d}s(\bm{ y })
\quad \mbx\in \bbR^2\setminus \overline{\p B}. \]
The mapping property of the single-layer potential shows that $SL_{\p B}$ is a linear and bounded operator from $ H^{-3/2}(\p B)$ to $H^2_{\mathrm{ loc } }(\bbR^2)$. Thus using the embedding theorem $L^2(\p B)\to H^{-3/2}(\p B)$ and trace properties of the functions in $H^2$, we see that $R$ is compact from $L^2(\p B)$ to $H^{3/2}(\p D)\times H^{1/2}(\p D)$. From the asymptotic behavior  of $G$, it follows that $H^{*}\phi$ coincides with the far-field pattern of the single-layer potential of biharmonic wave equation $- 2 \kappa^2 SL_{\p B}\phi$, which gives
\begin{equation*}
	H_{\p B}^{*}=\mcB_{{\rm Dir},D} R.
\end{equation*}
Accordingly, we can apply (1) of Theorem \ref{thm:MM} with $F=-F_D$, $G=\mcB_{{\rm Dir},D}$, $T=-2 \kappa^2\mcS_{{\rm Dir},D}^{*}$, $T_0=-2 \kappa^2\tilde{\mcS}_{{\rm Dir},D}$, $\tilde{F}=H_{\p B}^{*} H_{\p B}$, $\tilde{ G}=H_{\p B}^{*}$, $\tilde{T}=I$, $R=R$ and  deduce that \begin{equation*}
	H_{\p B}^{*}H_{\p B} \leq_{\mathrm{ fin} } - \mathrm{Re}\;F_D.
\end{equation*}

Let $B \not\subset D$. There exists $\G\subset \p B\setminus \overline{D}$ such that $\G$ is relatively open set and $\bbR^2\setminus \overline{\G \cup D}$ is connected. Set $W\coloneqq{\rm Ran}(H_{\G}^{*})\subset {\rm Ran}(H_{\p B}^{*})$. By Theorem \ref{thm:H-adjoint}, $W$ is infinite-dimensional and $W \cap {\rm Ran}(\mcB_{{\rm Dir},D})=\{ 0\}$. So we can apply  (2) of Theorem \ref{thm:MM} with $F=-F_D$, $G=\mcB_{{\rm Dir},D}$, $T=-2 \kappa^2\mcS_{{\rm Dir},D}^{*}$, $\tilde{F}=H_{\p B}^{*} H_{\p B}$, $\tilde{ G}=H_{\p B}^{*}$, $\tilde{T}=I$, $\tilde{T}_0=I$ and $W=W$, and  deduce that 
\begin{equation*}
	H_{\p B}^{*}H_{\p B} \not\leq_{\mathrm{ fin} } - \mathrm{Re}\;F_D
\end{equation*}
holds. Thus this finishes the proof of the theorem.
\end{proof}

\begin{theorem}\label{thm:MM-n}
Assume that $B\subset \bbR^2$ is a bounded and smooth domain.
If $D$ is a Neumann obstacle, then the following holds:
\begin{equation*}
	B \subset D \iff H_{\p B}^{*} H_{\p B} \leq_{\mathrm{ fin} }  \mathrm{Re}\;F_D.
\end{equation*}
\end{theorem}
\begin{proof}
Assume  $B \subset D$. Define $R\colon L^2(\p B)\to H^{-3/2}(\p D)\times H^{-1/2}(\p D)$ by $R \phi = (N v|_{\p D}, M v |_{\p D})^{\top}$,
where  $v=- 2 \kappa^2 SL_{\p B}\phi$ is the single-layer potential.
We use the mapping property of the single-layer potential, 
the embedding theorem $L^2(\p B)\to H^{-3/2}(\p B)$ and trace properties of the functions in $H^2$ to obtain that $R$ is compact from $L^2(\p B)$ to $H^{-3/2}(\p D)\times H^{-1/2}(\p D)$. Noting that $H^{*}\phi$ is indeed the far-field pattern of
$- 2 \kappa^2 SL_{\p B}\phi$ by the asymptotic behavior of  $G$, we  conclude that
\begin{equation*}
	H_{\p B}^{*}=\mcB_{{\rm Neu},D} R.
\end{equation*}
Applying (1) of Theorem \ref{thm:MM} with $F=F_D$, $G=\mcB_{{\rm Neu},D}$, $T=2 \kappa^2\mcS_{{\rm Neu},D}^{*}$, $T_0=2 \kappa^2\tilde{\mcS}_{{\rm Neu},D}$, $\tilde{F}=H_{\p B}^{*} H_{\p B}$, $\tilde{ G}=H_{\p B}^{*}$, $\tilde{T}=I$ and $R=R$, results that \begin{equation*}
	H_{\p B}^{*}H_{\p B} \leq_{\mathrm{ fin} } \mathrm{Re}\;F_D
\end{equation*}
holds.

Let $B \not\subset D$. There exists $\G\subset \p B\setminus \overline{D}$ such that $\G$ is relatively open set and $\bbR^2\setminus \overline{\G \cup D}$ is connected. Setting $W\coloneqq{\rm Ran}(H_{\G}^{*})\subset {\rm Ran}(H_{\p B}^{*})$, which is infinite-dimensional and satisfies $W \cap {\rm Ran}(\mcB_{{\rm Neu},D})=\{ 0\}$ by Theorem \ref{thm:H-adjoint}, and  applying  (2) of Theorem \ref{thm:MM} with $F=F_D$, $G=\mcB_{{\rm Neu},D}$, $T=2 \kappa^2\mcS_{{\rm Neu},D}^{*}$, $\tilde{F}=H_{\p B}^{*} H_{\p B}$, $\tilde{ G}=H_{\p B}^{*}$, $\tilde{T}=I$, $\tilde{T}_0=I$ and $W=W$, leads to 
\begin{equation*}
	H_{\p B}^{*}H_{\p B} \not\leq_{\mathrm{ fin} }  \mathrm{Re}\;F_D.
\end{equation*}
Therefore, the proof of the theorem is finished.
\end{proof}

Let $B$ be a sampling domain and $D$ the Dirichlet (or Neumann) obstacle. As suggested in \cite{albicker2020monotonicity}, we define the indicator functions $W_3$ and $W_4$ as follows:
\begin{align*}
	W_3 (B) &\coloneqq \# \{\lambda>0\colon  \lambda \,\, \text{is an eigenvalue of}\,\, \mathrm{Re}\; F_D +  H_{\p B}^{*}H_{\p B} \}\\
	W_4 (B) &\coloneqq \# \{\lambda > 0\colon  \lambda \,\, \text{is an eigenvalue of}\,\, -\mathrm{Re}\; F_D +  H_{\p B}^{*}H_{\p B} \}.
\end{align*}
By Theorem \ref{thm:MM-d}  for the Dirichlet obstacle and Theorem \ref{thm:MM-n} for the Neumann obstacle, $W_3(B)$ and $W_4(B)$ take smaller values when $B\subset D$ and larger values when $B\not \subset D$, which will be shown by some numerical examples. 
Moreover, compared with the factorization method, the monotonicity method does not require the exclusion of the corresponding transmission eigenvalues. 
\section{Numerical results}
\label{sec:numeri}
Using the shape characterization for the obstacle $D$ provided by Theorems \ref{thm:indicator-d} and \ref{thm:MM-d}, we present the numerical implementation of the factorization and monotonicity methods, and report numerical results for the reconstruction of the clamped obstacles with various shapes using far-field data to show the performance of these methods.

We generate synthetic far-field data by solving the direct problem numerically via the boundary integral equation method from \cite{dong2024novel}.
For $N=64$ incident directions $\mbd_j=(\cos (2\pi(j-1)/N), \sin (2\pi(j-1)/N))^\top$ $j=1,2,\cdots N$, $u^{\infty}(\mbd_i,\mbd_j )$ is obtained at $N$ observation directions $\mbd_i$ $i=1,2,\cdots N$. For the far-field operator $F_D$, we employ the trapezoid rule to derive its discrete version \begin{equation*}
\bm{ F }= \frac{2\pi}{N} \left[ u^{\infty}(\mbd_i,\mbd_j ) \right]_{i,j=1}^N \in \bbC^{N \times N}.
\end{equation*}
To simulate the noisy data, we consider the following random perturbation $\bm{ F }_{\delta}$  of the matrix $\bm{ F }$ 
\begin{equation} 
	\notag 
	\bm{ F }_{\delta} =  \left[  \bm{ F }_{ij}  
	\left( 1 
	+ \delta  \frac{\bm{ \xi}_{1,ij}+\mathrm{ i } \bm{\xi}_{2,ij}}{\| \bm{ \xi}_{1}+\mathrm{ i } \bm{\xi}_{2}\|_2}  
	\right)  \right]_{i,j=1}^N ,
	\end{equation}
	where $\delta$ is the noise level of the data, and $\bm{ \xi}_1 \in  \bbR^{N \times N} $ and $\bm{ \xi}_2 \in \bbR^{N \times N}$ are {real-valued matrices whose elements are
random numbers from the standard normal distributions.} 
Let $\phi_{\mbz}$ be approximated by the vector $\phi_{\mbz}=(e^{-\mathrm{ i } \kappa \mbd_j \cdot  \mbz})_{j=1}^N$. Let  $(\lambda_{1,j},\psi_{1,j},\tilde{\psi}_{1,j})_{j=1}^N$ and $(\lambda_{2,j},\psi_{2,j})_{j=1}^N$ denote the singular system of $\bm{F}_{\delta} $ and the eigensystem of $(\bm{ F }_{\delta})_{\#}$, respectively.
Following \cite{kirsch1998charac}, to deal with the influence of noisy data, we employ the following regularized indicator functions:
\begin{align*}
W_{1,N}(\mbz )&= \left[
\sum\limits_{j=1}^{N} \dfrac{| \lambda_{1,j} | }{\alpha +| \lambda_{1,j} |^{2} }                  |( \phi_{\mbz} ,\psi_{1,j}) |^2 \right]^{-1}, \\
W_{2,N}(\mbz )&= \left[
\sum\limits_{j=1}^{N} \dfrac{| \lambda_{2,j} | }{\alpha +| \lambda_{2,j} |^{2} }                  |( \phi_{\mbz} ,\psi_{2,j}) |^2 \right]^{-1},
\end{align*}
for the factorization methods, where $\alpha$ is the regularization parameter, which can be chosen by the discrepancy principle. By virtue of the indicator functions, we propose a numerical reconstruction algorithm:
\begin{algorithm}[H]
\caption{Factorization Methods (FM)}
\begin{algorithmic}[1]
	\STATE Select a grid of sampling points $\mbz$ in  $\O$ containing the obstacle $D$.
	
	\STATE Compute the indicator function $W_{1,N}(\mbz)$ (or $W_{2,N}(\mbz)$) for each sampling point $\mbz\in \O$.
	
	\STATE  Plot the contours of the indicator function $W_{1,N}(\mbz)$ (or $W_{2,N}(\mbz)$).
	
	\STATE Determine the obstacle $D$ according to the region where  $W_{1,N}(\mbz)$ (or $W_{2,N}(\mbz)$) attains large values.
\end{algorithmic}
\end{algorithm}

Following \cite{albicker2020monotonicity}, we choose $B=B_{h/2}(\mbz)$ and for the operator $ H_{\p B}^{*}H_{\p B}$, we also adopt the  trapezoid rule to derive its discrete version
\begin{equation*}
\bm{ T }_{B}=
\frac{2\pi}{N} \left[\pi h e^{\mathrm{ i } \kappa \mbz\cdot  (\bm{d}_j- \bm{d}_i)}J_0(\frac{\kappa h}{2}| \bm{d}_i-\bm{d}_j | ) \right]_{i,j=1}^N \in  \bbC^{N \times N}.
\end{equation*}
Then we obtain the discrete version of the operator  $\pm\mathrm{Re}\;F_D+  H_{\p B}^{*}H_{\p B}$:
\begin{equation*}
\bm{ A }_{\pm} = \pm \mathrm{Re}\; \bm{ F } +   \bm{ T }_{B}.
\end{equation*}
Moreover, for the noisy data $\bm{ F }_{\delta}$,  we consider the  following random perturbation $\bm{ A }_{\pm,\delta} $ of the matrix $\bm{ A }_{\pm} $:
\begin{equation*}
\bm{ A }_{\pm,\delta} = \pm \mathrm{Re}\;\bm{ F }_{\delta} + \bm{ T }_{B},
\end{equation*}
where $\mathrm{Re}\;\bm{ F }_{\delta} = (\bm{ F }_{\delta}+\bm{ F }_{\delta}^{*})/2$.
For the monotonicity method, the discrete indicator function is given by
\begin{align*}
W_{3,N} (\mbz) &= \# \{\lambda> \tilde{\delta} \colon  \lambda  \,\, \text{is an eigenvalue of}\,\, \bm{ A }_{+,\delta}\},\\
W_{4,N} (\mbz) &= \# \{\lambda> \tilde{\delta} \colon  \lambda \,\, \text{is an eigenvalue of}\,\, \bm{ A }_{-,\delta}\},
\end{align*}
where $\tilde{\delta}>0$ is a threshold parameter and its selection can be found in Section 6.2 of \cite{albicker2020monotonicity},
and we have the following numerical reconstruction algorithm based on the monotonicity method:
\begin{algorithm}[H]
\caption{Monotonicity Method (MM)}
\begin{algorithmic}[1]
	\STATE Select a grid of sampling points $\mbz$ in  $\O$ containing the obstacle $D$.
	
	\STATE Compute the indicator function $W_{3,N}(\mbz)$ (or  $W_{4,N}(\mbz)$) for each sampling point $\mbz\in \O$.
	
	\STATE  Plot the contour of the indicator function $W_{3,N}(\mbz)$  (or  $W_{4,N}(\mbz)$).
	
	\STATE Determine the Dirichlet (or Neumann) obstacle $D$ according to the region where $W_{3,N}(\mbz)$  (or  $W_{4,N}(\mbz)$)   attains small values.
\end{algorithmic}
\end{algorithm}

\begin{example}\label{exam:circle}
Our first numerical example addresses the reconstruction of a single obstacle using noise-free far-field data. Specifically, $D$ is a circular obstacle of radius $0.4$ centered at the origin and its exact shape is represented by the dashed line in Figure \ref{fig:circular}.
The sampling domain $\O=[-4,4]\times[-4,4]$ is discretized using a uniform grid with $81$ points in each direction. We choose $h=0.1$, the regularization parameter $\alpha=0$ and the thresholding parameter $\tilde{\delta}=0$. We observe that the reconstruction results, 
shown in Figure \ref{fig:circular} for  wavenumber $\kappa=2\pi$, are in accordance with the exact shape of the obstacle.
\begin{figure}[htbp!]
	\centering
	\begin{subfigure}[b]{0.31\textwidth}
		\centering
		\includegraphics[width=\textwidth]{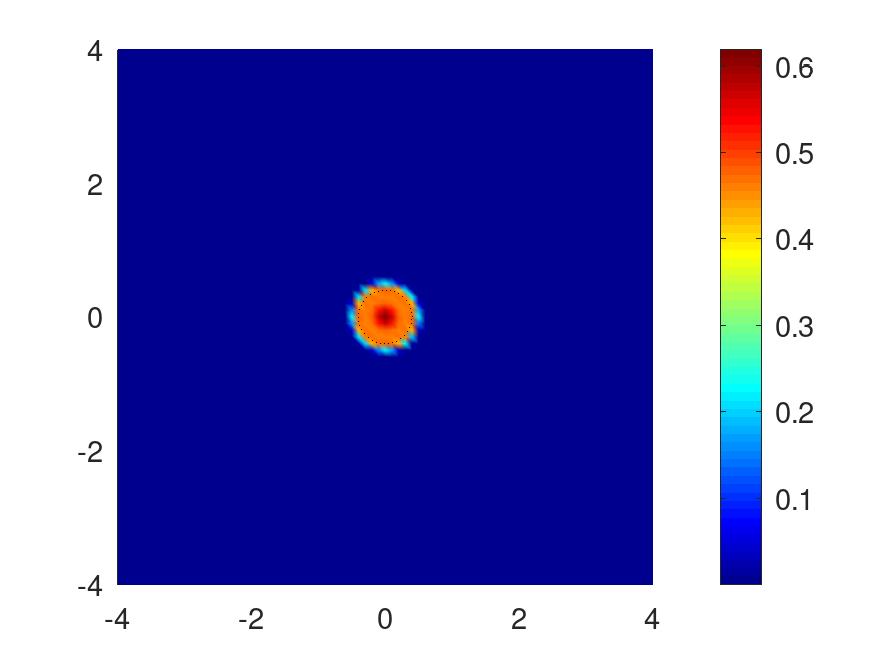}
		\caption{$W_1$}
		\label{fig:circular-1}
	\end{subfigure}
	\hfill
	\begin{subfigure}[b]{0.31\textwidth}
		\centering
		\includegraphics[width=\textwidth]{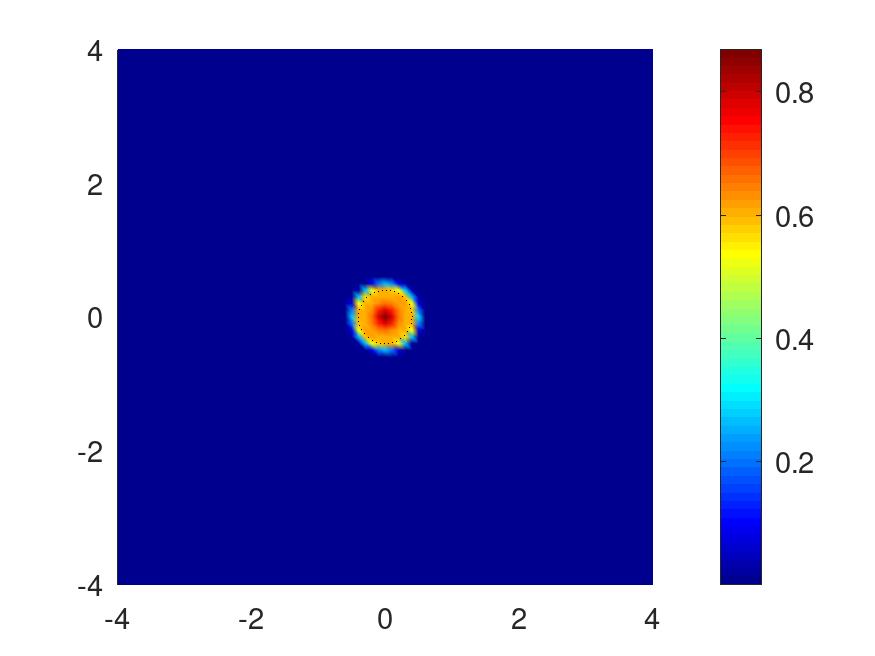}
		\caption{$W_2$}
		\label{fig:circular-2}
	\end{subfigure}
	\hfill
	\begin{subfigure}[b]{0.31\textwidth}
		\centering
		\includegraphics[width=\textwidth]{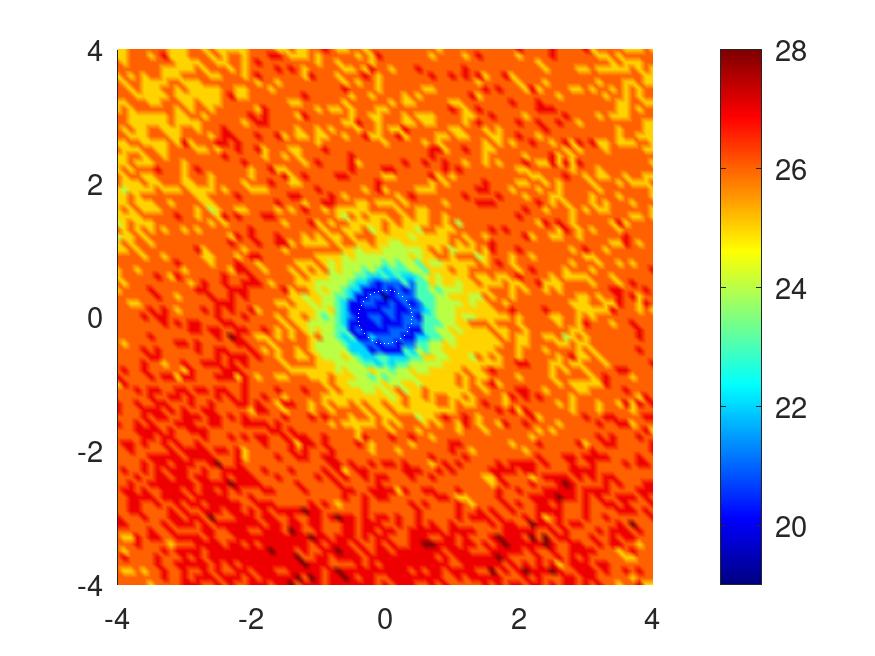}
		\caption{$W_3$}
		\label{fig:circular-3}
	\end{subfigure}
	\caption{Example \ref{exam:circle}: the reconstruction of a circular obstacle}
	\label{fig:circular}
\end{figure}
\end{example}

\begin{example}\label{exam:ellipse}
We next examine reconstruction near a numerically approximated clamped transmission eigenvalue.
The impenetrable obstacle $D$ is an ellipse, whose parameterization is
\begin{equation*}
	\mbx(t)= (\cos(t),  0.5\sin(t)  )^\top \quad t \in [0,2\pi],
\end{equation*}
and  its exact shape is described by the dashed line.
We note that  the fourth clamped transmission eigenvalue  is approximately  $ 5.36324$ for the ellipse $D$ according to Section 5.2 in \cite{harris2025transmission}.
Likewise, the sampling domain is $ \O= [-4, 4]\times[-4, 4]$, which is discretized using a uniform grid with $81$ points in each direction. We again choose $h = 0.1$, the regularization parameter $\alpha=0$ and the thresholding parameter  $\tilde{\delta}=0$. In Figure \ref{fig:ellipse}, we present the reconstruction results by noise-free far-field data when the wavenumber $\kappa$ is  $5.36324$. It is observed that the factorization methods cannot detect the obstacle at a transmission eigenvalue, whereas the monotonicity method remains effective.
\begin{figure}[htbp!]
	\centering
	\begin{subfigure}[b]{0.31\textwidth}
		\centering
		\includegraphics[width=\textwidth]{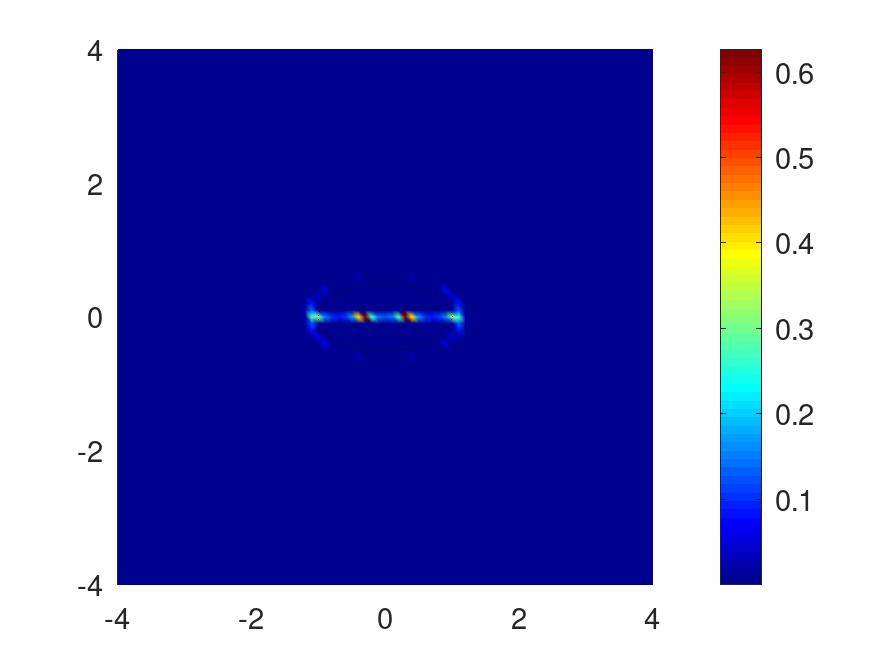}
		\caption{$W_1$}
		\label{fig:ellipse-1}
	\end{subfigure}
	\hfill
	\begin{subfigure}[b]{0.31\textwidth}
		\centering
		\includegraphics[width=\textwidth]{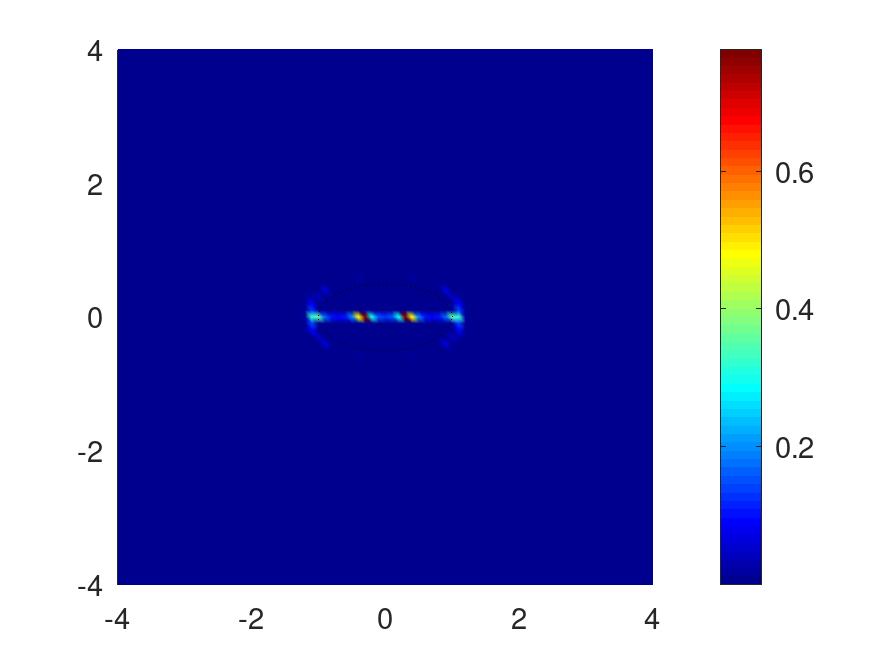}
		\caption{$W_2$}
		\label{fig:ellipse-2}
	\end{subfigure}
	\hfill
	\begin{subfigure}[b]{0.31\textwidth}
		\centering
		\includegraphics[width=\textwidth]{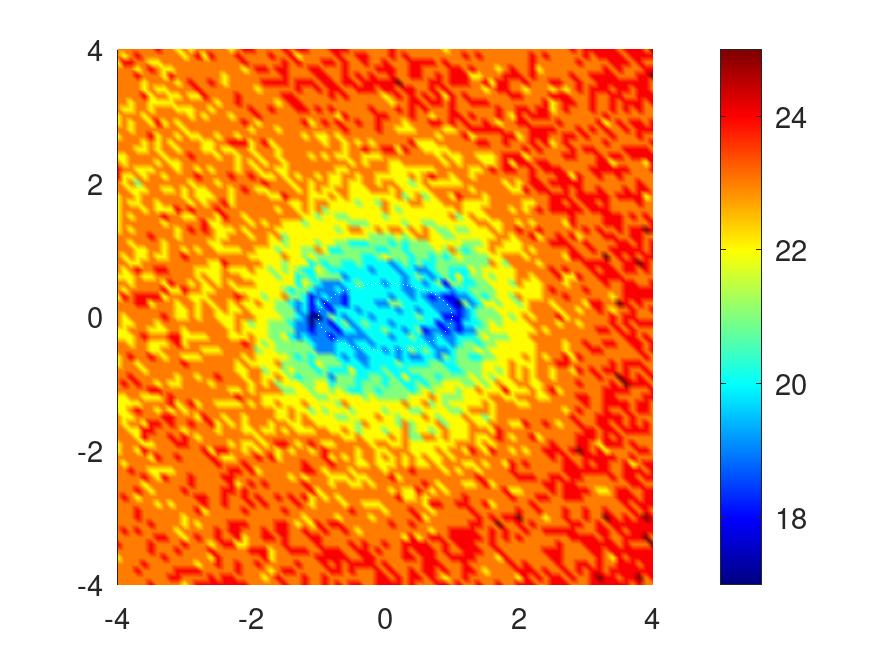}
		\caption{$W_3$}
		\label{fig:ellipse-3}
	\end{subfigure}
	\caption{Example \ref{exam:ellipse}: the reconstruction of an elliptical obstacle}
	\label{fig:ellipse}
\end{figure}
\end{example}

\begin{example}\label{exam:square}
In the third example, we study the influence of noise on the reconstruction quality. The impenetrable obstacle $D$ is  a rounded square, whose parameterization is
\begin{equation*}
	\mbx(t)= 0.25(\cos^3(t)+\cos(t), \sin^3(t) + \sin(t)  )^\top \quad t \in [0,2\pi],
\end{equation*}
and its exact shape is described by the dashed line. The sampling domain $ \O= [-4, 4]\times[-4, 4]$  is discretized using a uniform grid with $81$ points in each direction. We choose $h = 0.1$, the regularization parameter $\alpha=10^{-6}$ and the thresholding parameter  $\tilde{\delta}=10^{-14}$. Figure \ref{fig:square} illustrates the reconstruction results of the three indicator functions from  noise-free, $1\%$ and $5\%$ data with $\kappa=2\pi$.
As illustrated in Figure \ref{fig:square}, our algorithms are robust to random noise. Furthermore, the factorization method can achieve higher reconstruction quality than the monotonicity method.
\begin{figure}[htbp!]
	\centering
	\begin{subfigure}[b]{0.31\textwidth}
		\centering
		\includegraphics[width=\textwidth]{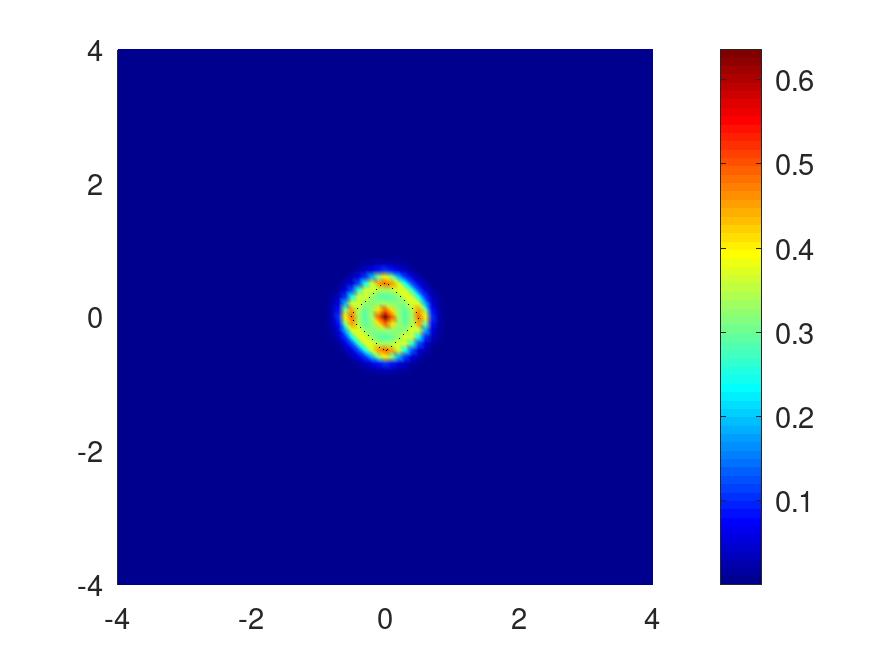}
		\caption{No noise}
		\label{fig:square-1-1}
	\end{subfigure}
	\hfill
	\begin{subfigure}[b]{0.31\textwidth}
		\centering
		\includegraphics[width=\textwidth]{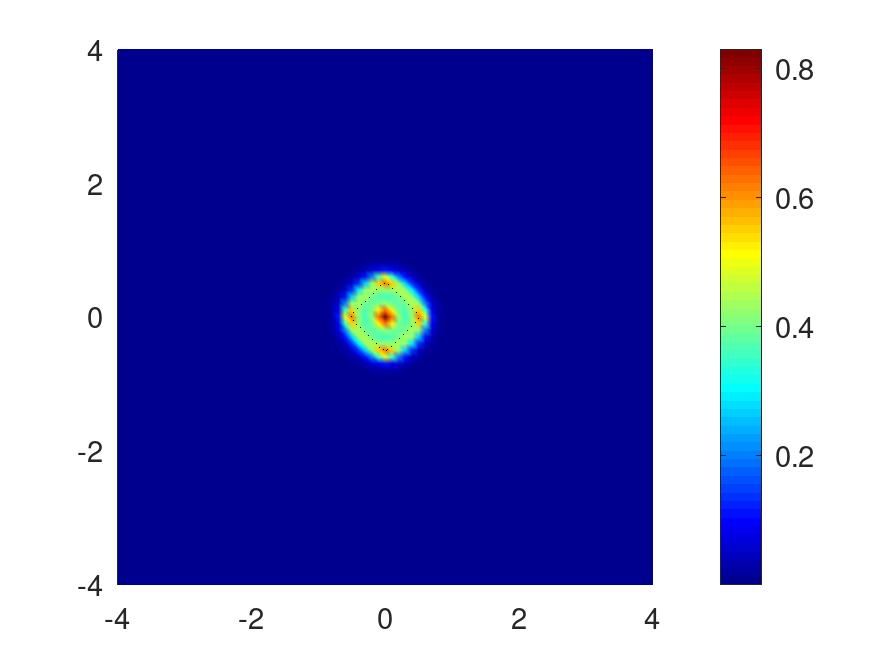}
		\caption{No noise}
		\label{fig:square-2-1}
	\end{subfigure}
	\hfill
	\begin{subfigure}[b]{0.31\textwidth}
		\centering
		\includegraphics[width=\textwidth]{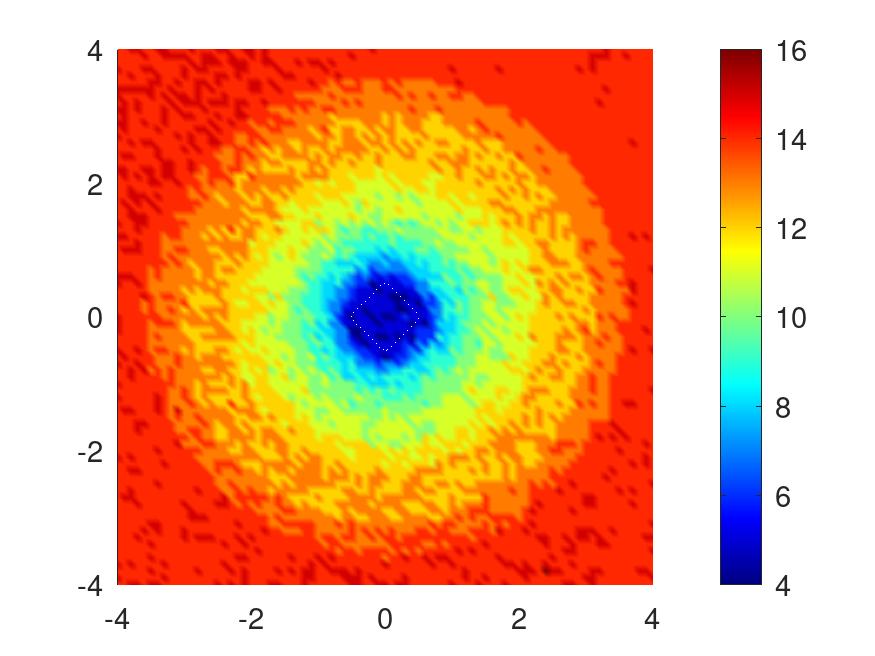}
		\caption{No noise}
		\label{fig:square-3-1}
	\end{subfigure}
	
	\begin{subfigure}[b]{0.31\textwidth}
		\centering
		\includegraphics[width=\textwidth]{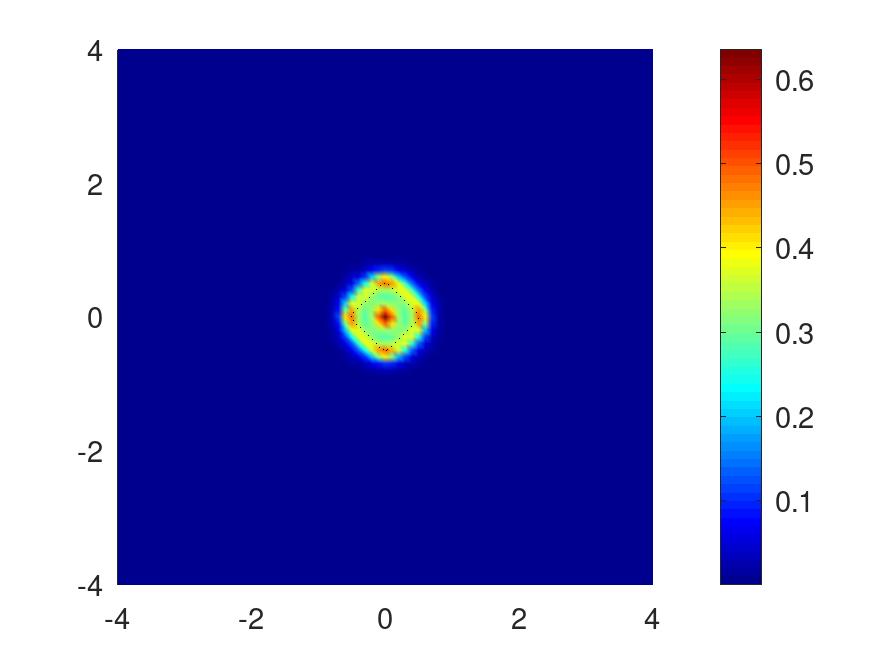}
		\caption{$1\%$ noise}
		\label{fig:square-1-2}
	\end{subfigure}
	\hfill
	\begin{subfigure}[b]{0.31\textwidth}
		\centering
		\includegraphics[width=\textwidth]{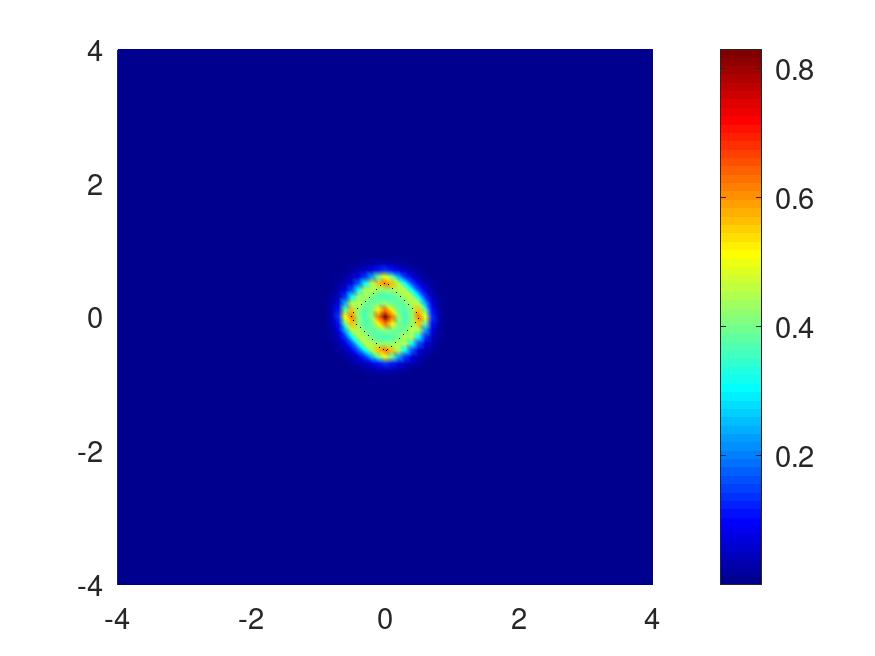}
		\caption{$1\%$ noise}
		\label{fig:square-2-2}
	\end{subfigure}
	\hfill
	\begin{subfigure}[b]{0.31\textwidth}
		\centering
		\includegraphics[width=\textwidth]{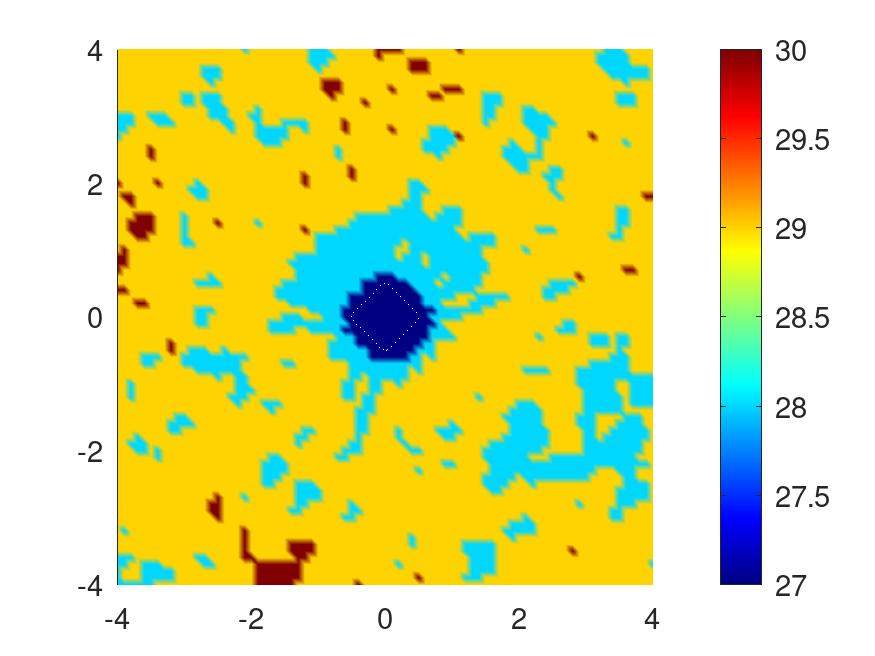}
		\caption{$1\%$ noise}
		\label{fig:square-3-2}
	\end{subfigure}
	
	\begin{subfigure}[b]{0.31\textwidth}
		\centering
		\includegraphics[width=\textwidth]{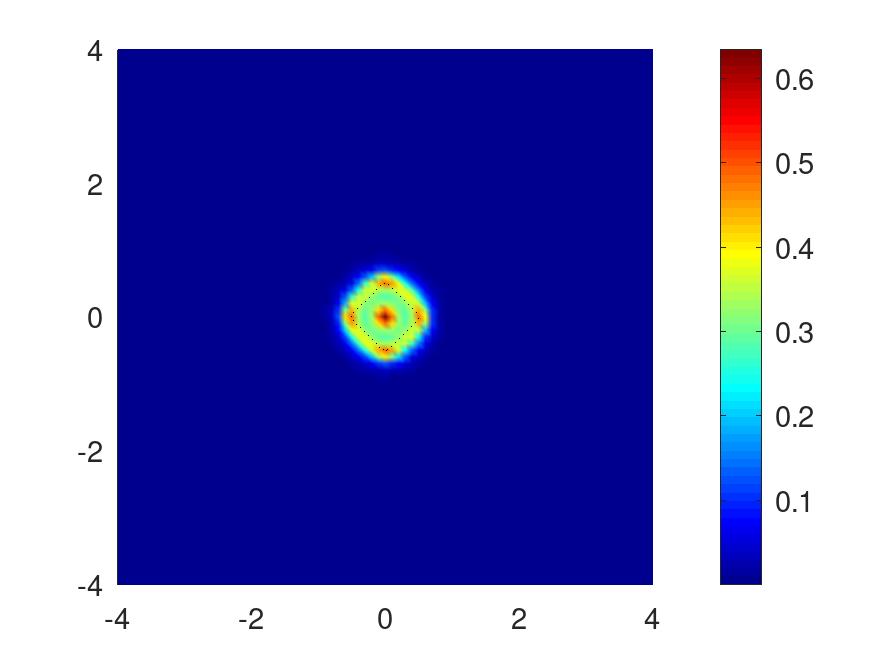}
		\caption{$5\%$ noise}
		\label{fig:square-1-3}
	\end{subfigure}
	\hfill
	\begin{subfigure}[b]{0.31\textwidth}
		\centering
		\includegraphics[width=\textwidth]{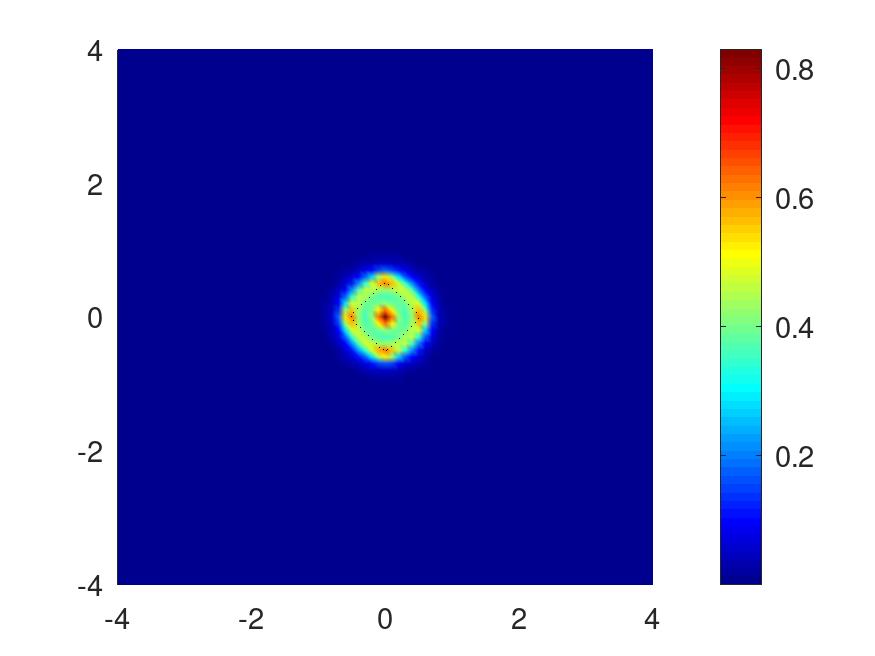}
		\caption{$5\%$ noise}
		\label{fig:square-2-3}
	\end{subfigure}
	\hfill
	\begin{subfigure}[b]{0.31\textwidth}
		\centering
		\includegraphics[width=\textwidth]{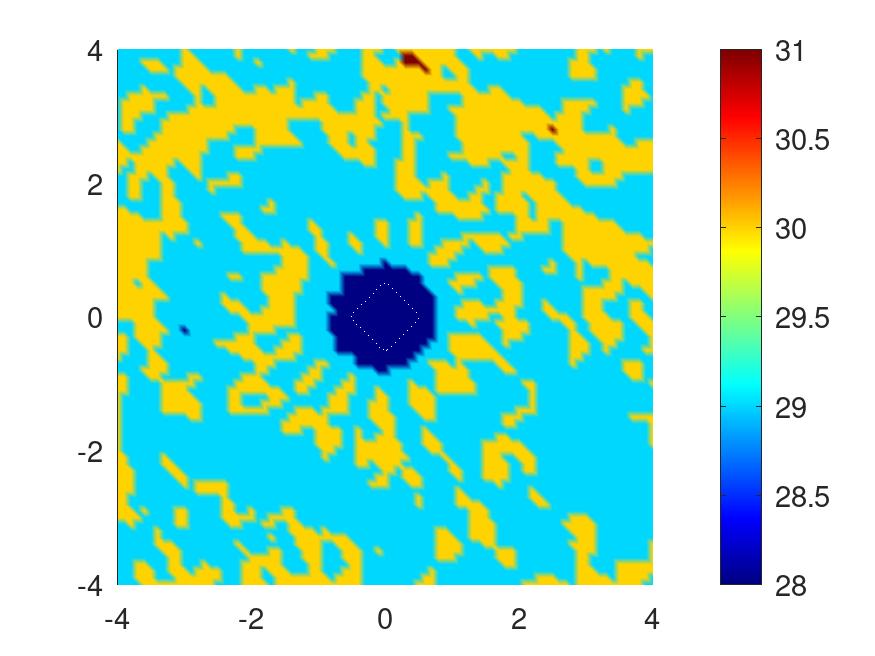}
		\caption{$5\%$ noise}
		\label{fig:square-3-3}
	\end{subfigure}
	\caption{ Example \ref{exam:square}: the $j-$column corresponds to the reconstruction results of the indicator function $W_j$ ($j=1,2,3$).}
	\label{fig:square}
\end{figure}
\end{example}

\begin{example}\label{exam:two}
The fourth numerical example deals with the case of two  Dirichlet obstacles, which consist of a rounded triangle and a kite domain. The parameterizations of the two obstacles are given by
\begin{align*}
	\mbx_1(t) & =  (-4,-3)^\top + (0.8+0.12\cos(3t))( \cos(t), \sin t)^\top \quad t \in [0,2\pi], \\
	\mbx_2(t)&= (3,4)^\top +  0.5(0.65 \cos(2t)+\cos t-0.65, 1.5\sin t)^\top  \quad t \in [0,2\pi],
\end{align*}
and their exact shapes are described by the  dashed lines.
The sampling domain $ \O= [-10, 10]\times[-10, 10]$  is discretized using a uniform grid with $201$ points in each direction. We choose $h = 0.1$, the regularization parameter $\alpha=10^{-6}$ and the thresholding parameter  $\tilde{\delta}=10^{-14}$. The reconstruction results in Figure \ref{fig:two} are obtained at the wavenumber $\kappa = 2\pi$ under noise levels of  $0\%$, $1\%$ and $5\%$.
The numerical results demonstrate that the proposed algorithms can yield satisfactory reconstructions  for the two obstacles.
\begin{figure}[htbp!]
	\centering
	\begin{subfigure}[b]{0.31\textwidth}
		\centering
		\includegraphics[width=\textwidth]{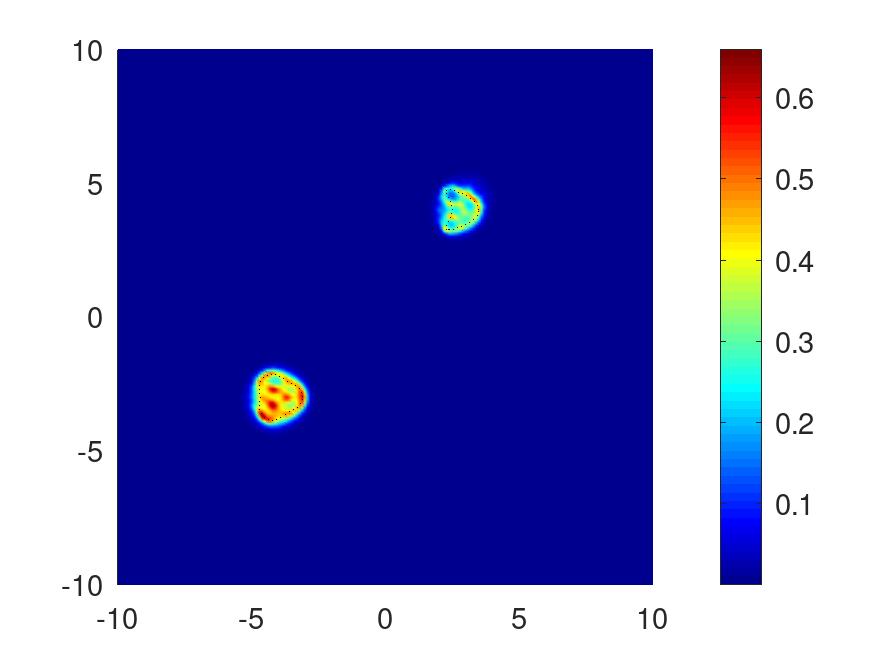}
		\caption{No noise}
		\label{fig:two-1-1}
	\end{subfigure}
	\hfill
	\begin{subfigure}[b]{0.31\textwidth}
		\centering
		\includegraphics[width=\textwidth]{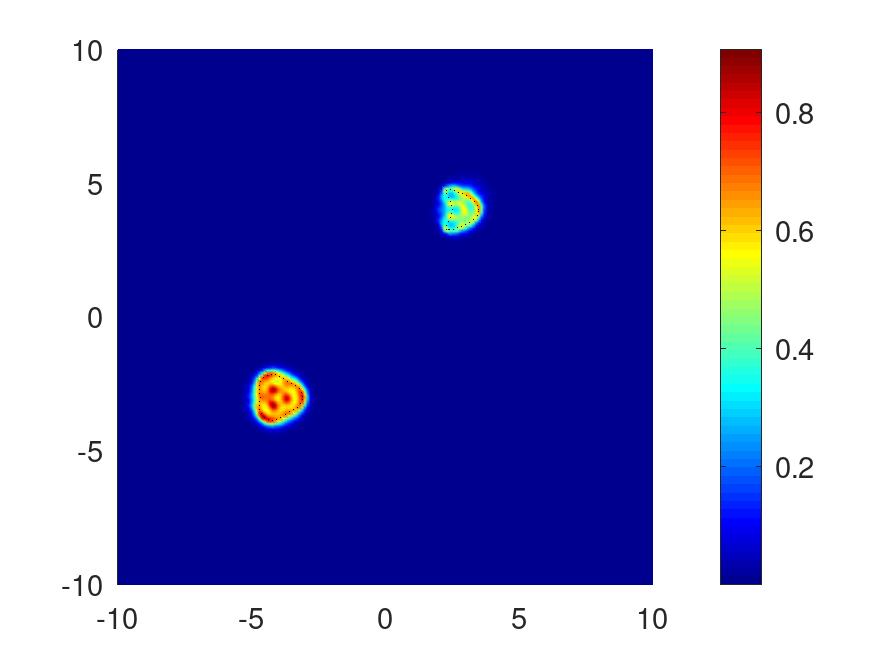}
		\caption{No noise}
		\label{fig:two-2-1}
	\end{subfigure}
	\hfill
	\begin{subfigure}[b]{0.31\textwidth}
		\centering
		\includegraphics[width=\textwidth]{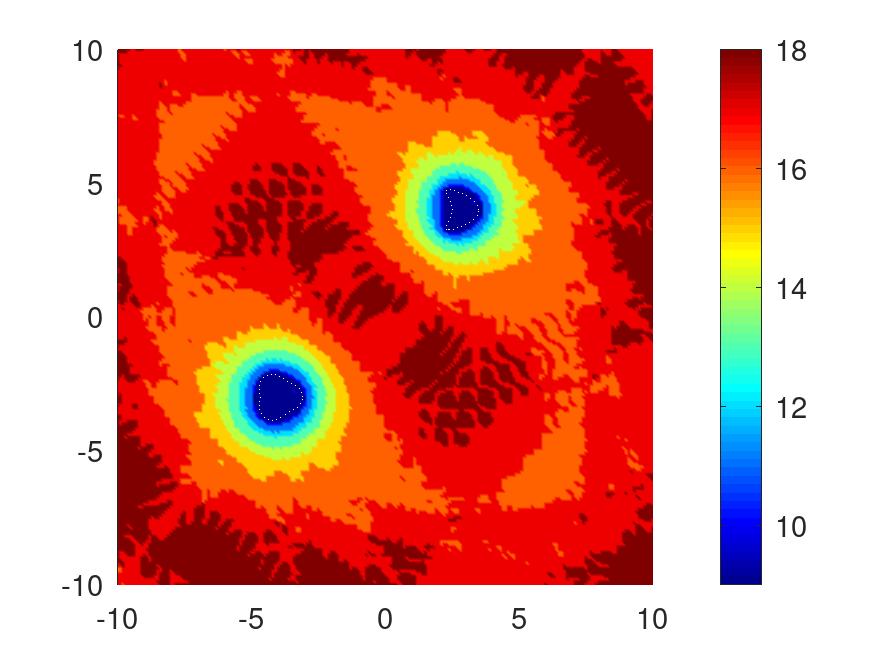}
		\caption{No noise}
		\label{fig:two-3-1}
	\end{subfigure}
	
	\begin{subfigure}[b]{0.31\textwidth}
		\centering
		\includegraphics[width=\textwidth]{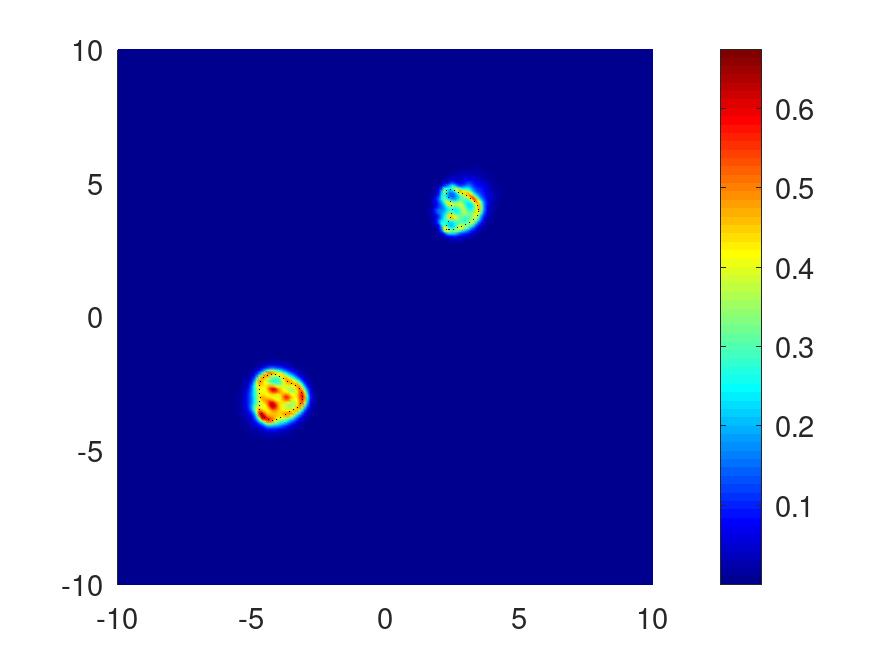}
		\caption{$1\%$ noise}
		\label{fig:two-1-2}
	\end{subfigure}
	\hfill
	\begin{subfigure}[b]{0.31\textwidth}
		\centering
		\includegraphics[width=\textwidth]{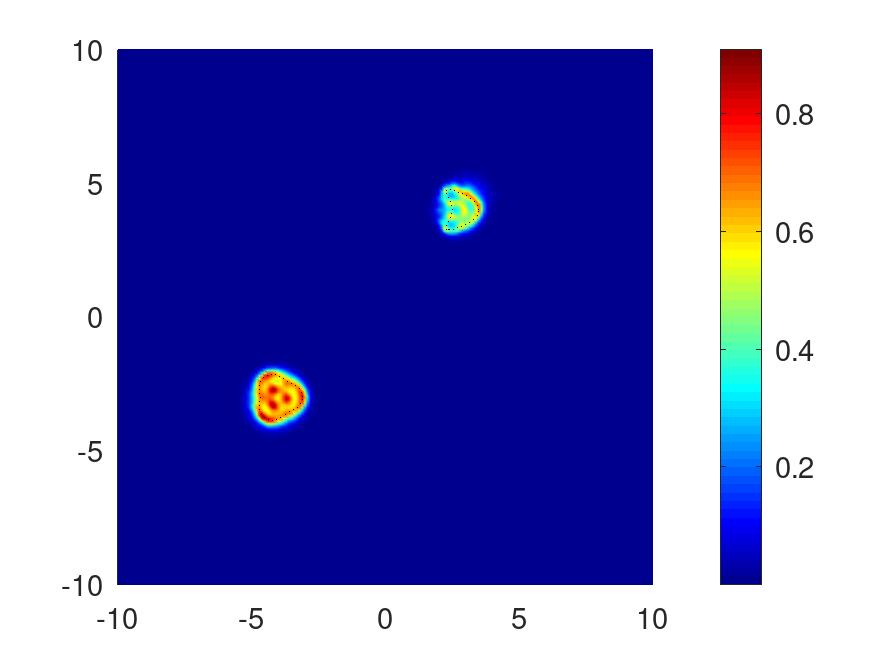}
		\caption{$1\%$ noise}
		\label{fig:two-2-2}
	\end{subfigure}
	\hfill
	\begin{subfigure}[b]{0.31\textwidth}
		\centering
		\includegraphics[width=\textwidth]{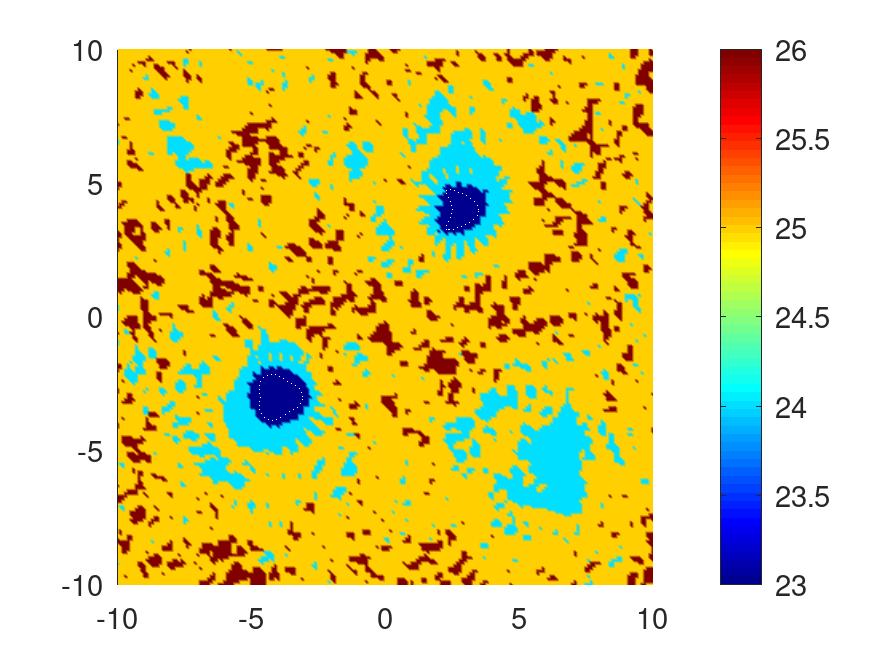}
		\caption{$1\%$ noise}
		\label{fig:two-3-2}
	\end{subfigure}
	
	\begin{subfigure}[b]{0.31\textwidth}
		\centering
		\includegraphics[width=\textwidth]{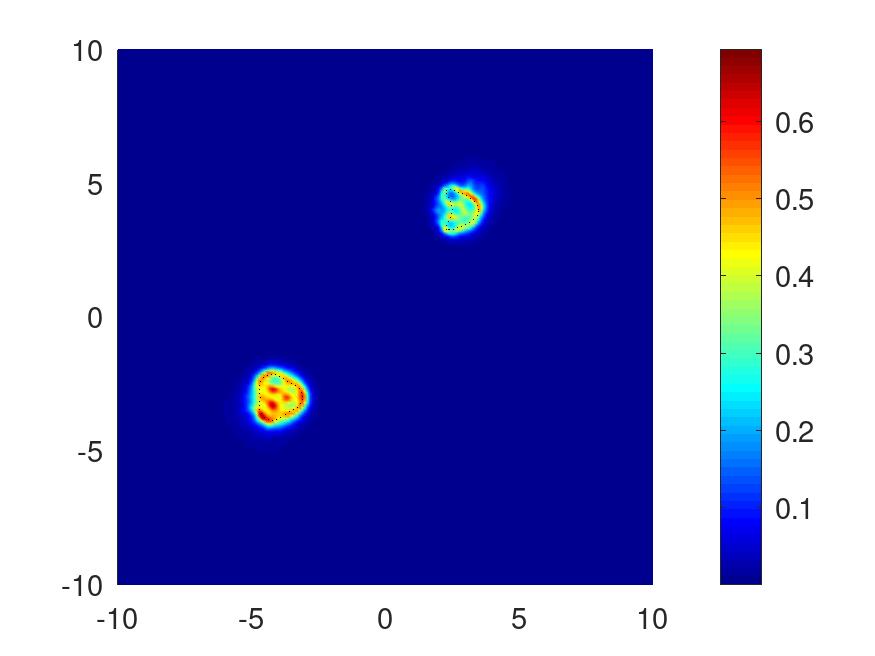}
		\caption{$5\%$ noise}
		\label{fig:two-1-3}
	\end{subfigure}
	\hfill
	\begin{subfigure}[b]{0.31\textwidth}
		\centering
		\includegraphics[width=\textwidth]{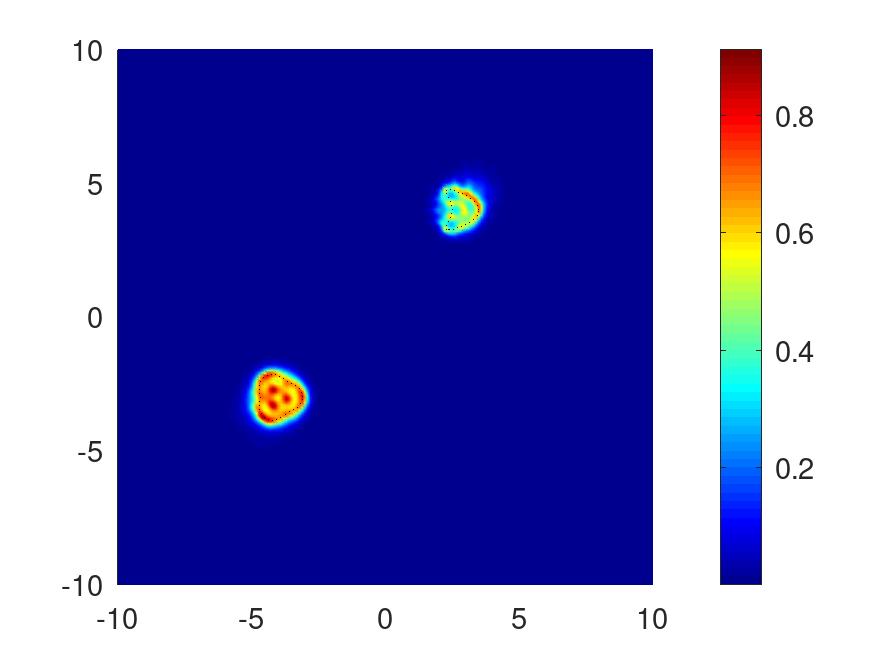}
		\caption{$5\%$ noise}
		\label{fig:two-2-3}
	\end{subfigure}
	\hfill
	\begin{subfigure}[b]{0.31\textwidth}
		\centering
		\includegraphics[width=\textwidth]{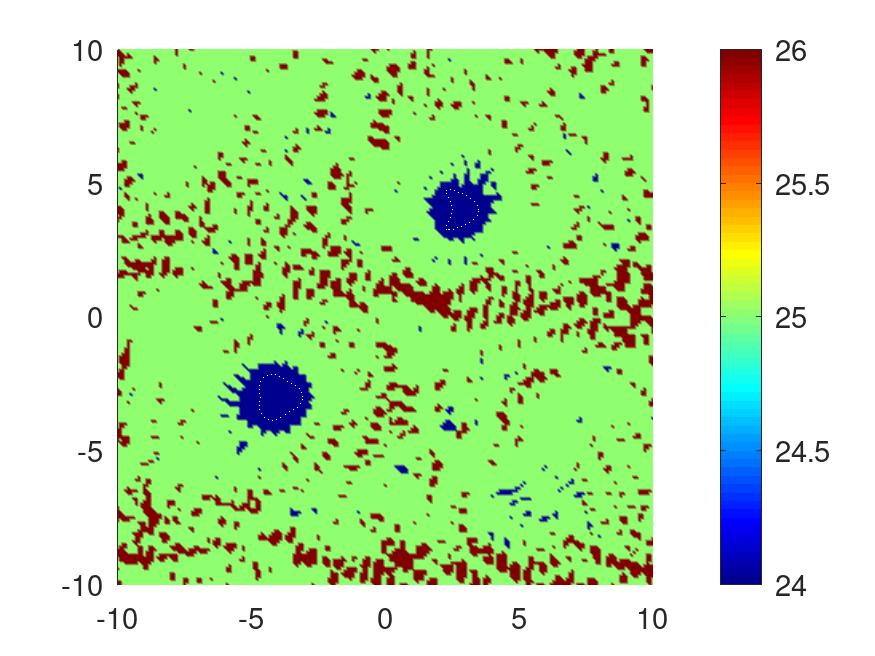}
		\caption{$5\%$ noise}
		\label{fig:two-3-3}
	\end{subfigure}
	\caption{ Example \ref{exam:two}: the $j-$column corresponds to the reconstruction results of the indicator function $W_j$ ($j=1,2,3$).}
	\label{fig:two}
\end{figure}
\end{example}

\section*{Acknowledgments}

This work is supported by the National Natural Science Foundation of China(No. 12371393 and 11971150), Natural Science Foundation of Henan Province(No. 242300421047), Natural Science Foundation of Xinjiang Uygur Autonomous Region(No. 2024D01C227) and  Science and Technology Program of Xinjiang Uyghur Autonomous Region(No. 2025A03011-3).

\end{document}